\documentclass[11pt]{article}
\usepackage{a4wide,amsmath,amsbsy,amsfonts,amssymb,
stmaryrd,amsthm,mathrsfs,graphicx, amscd, tikz-cd, cancel, supertabular}
\usepackage[normalem]{ulem}
\usetikzlibrary{matrix,arrows,decorations.pathmorphing}

\usepackage{xr}

\usepackage{subcaption}

\usepackage{tikz}
\usetikzlibrary{matrix,arrows,decorations.pathmorphing}
\usepackage[all,cmtip]{xy}
\usepackage{qtree}

\usepackage[top=1.2in,bottom=1.2in,left=1.21in,right=1.21in]{geometry}
\usepackage[
	hypertexnames=false,
	hyperindex,
	pagebackref,
	pdftex,
	breaklinks=true,
	bookmarks=false,
	colorlinks,
	linkcolor=blue,
	citecolor=red,
	urlcolor=red,
]{hyperref}
\usepackage{hyperref}

\newcommand\bs{\backslash}

\newtheorem{theorem}{Theorem}[section]
\newtheorem{proposition}[theorem]{Proposition}
\newtheorem{corollary}[theorem]{Corollary}
\newtheorem{lemma}[theorem]{Lemma}

\theoremstyle{definition}

\newtheorem{question}[theorem]{Question}

\numberwithin{equation}{section}

\theoremstyle{definition}
\newtheorem{remark}[theorem]{Remark}

\newcommand{\para}[1]{\medskip\noindent\textbf{#1.}}

\newcommand{\Cs}{\mathscr{C}}

\newcommand{\Lc}{\mathcal{L}}

\newcommand{\Bc}{\mathcal{B}}

\newcommand{\Oc}{\mathcal{O}}

\newcommand{\half}{\tfrac{1}{2}}

\newcommand{\Cb}{\mathbb{C}}
\newcommand{\C}{\mathbb{C}}

\newcommand{\Fb}{\mathbb{F}}

\newcommand{\G}{\Gamma}

\newcommand{\Pb}{\mathbb{P}}
\newcommand{\Qb}{\mathbb{Q}}
\newcommand{\Q}{\mathbb{Q}}
\newcommand{\Rb}{\mathbb{R}}
\newcommand{\Vb}{\mathbb{V}}
\newcommand{\Zb}{\mathbb{Z}}
\newcommand{\Z}{\mathbb{Z}}

\newcommand{\Sf}{\mathfrak{S}}
\newcommand{\Af}{\mathfrak{A}}
\newcommand{\Cf}{\mathfrak{C}}

\newcommand{\Df}{\mathfrak{D}}
\newcommand{\Vf}{\mathfrak{V}}

\newcommand{\Mcbar}{\overline{\mathcal{M}}}
\newcommand{\Mc}{\mathcal{M}}

\newcommand{\beq}{\begin{eqnarray}}
\newcommand{\eeq}{\end{eqnarray}}

\newcommand\ssm{\smallsetminus}

\newcommand{\ir}{{ir}}

\newcommand{\Hf}{\mathfrak{H}}

\DeclareMathOperator{\Aut}{Aut}

\DeclareMathOperator{\End}{End}

\DeclareMathOperator{\Hom}{Hom}

\DeclareMathOperator{\pic}{Pic}

\DeclareMathOperator{\Sp}{Sp}

\DeclareMathOperator{\PSL}{PSL}
\DeclareMathOperator{\SL}{SL}

\title{\vspace{-1in}Arithmeticity of the monodromy of the\\
 Wiman-Edge pencil}

\author{Benson Farb and Eduard Looijenga \thanks{The first author was supported in part by National Science Foundation Grant Nos. DMS-1105643 and DMS-1406209. The second author is supported by the Chinese National Science Foundation. Both authors are supported by the Jump Trading Mathlab Research Fund. }}

\begin{document}
\maketitle
\begin{abstract}
The {\em Wiman-Edge pencil} is the universal family $\Cs/\mathcal B$ of projective, genus $6$, complex-algebraic curves admitting a faithful action of the icosahedral group $\Af_5$.  The goal of this paper is to prove that the monodromy of $\Cs/\mathcal B$ is commensurable with a Hilbert modular group; in particular is arithmetic. We then give a modular interpretation of this, as well as a uniformization of $\mathcal B$.
 \end{abstract}

\tableofcontents

\section{Introduction}

 The {\em Wiman-Edge pencil} is the universal family $\Cs/\mathcal B$ of projective, genus $6$, complex-algebraic curves  admitting a faithful action of the icosahedral group $\Af_5$.   It has $5$ singular members; including a reducible curve of $10$ lines with intersection pattern the Petersen graph, and a union of $5$ conics with intersection pattern the complete graph on $5$ vertices.  Discovered by Wiman \cite{Wiman} (1895) and Edge \cite{Edge} (1981), the Wiman-Edge pencil appears in a variety of contexts, including: 
 
\begin{enumerate}
\item $\Cs/{\mathcal B}$ is a natural pencil of curves on the quintic del Pezzo surface $S$. It is invariant by the full automorphism group of $S$, i.e., the symmetric group of degree five, $\Sf_5$, with each $C_t\in{\mathcal B}$ being $\Af_5$-invariant, and with a unique smooth member $C_0$ that is $\Sf_5$-invariant, called the {\em Wiman curve}.

\item $\mathcal B$ is the moduli space of K3-surfaces with (a certain) faithful $\mu_2\times\Af_5$ action; see \S\ref{subsection:K3}.

\item $\Cs/\mathcal B$ is the quotient of one of the two 1-parameter families of lines on a nonsingular member of the Dwork pencil of Calabi-Yau quintic threefolds by it's group of automorphisms. 
\end{enumerate}

For a number of recent papers on the Wiman-Edge pencil, see \cite{Cheltsov, CKS, DFL, Za}.  

\bigskip

Given a family of varieties, it is a basic problem to compute its monodromy, to relate this to geometric properties of the family, and to use this information to uniformize (if possible) the base in terms of a period mapping, via Hodge structures.  While general theory has been developed around these questions, explicit computations can be quite difficult, and accordingly there are fewer of these.  The purpose of this paper is to solve these problems for the Wiman-Edge pencil $\Cs/\mathcal B$.  We prove that the monodromy of $\Cs/\mathcal B$ is commensurable with a Hilbert modular group; in particular that it is arithmetic. We then give a modular interpretation of this, and use it to uniformize $\mathcal B$.

Restricting to the smooth locus $\Cs/\mathcal B^\circ$, we obtain a family of smooth, genus $6$ curves, and so 
(choosing, say, the Wiman curve $C_0$ as representing the base point) a {\em monodromy} representation
\begin{equation}
\label{equation:monodromy:intro}
\rho:\pi_1({\mathcal B}^\circ)\to \Aut(H_1(C_0;\Zb))\cong \Sp_{12}(\Zb)
\end{equation}
that records how the fibers $C_t$ twist along loops in ${\mathcal B}^\circ$.  The isomorphism in \eqref{equation:monodromy:intro} comes from the fact that diffeomorphisms of $C_0$ preserve the algebraic intersection number on $C_0$, which is a symplectic pairing on $H_1(C_0;\Zb)$.  But the monodromy preserves more structure, for example it commutes with the $\Af_5$ action on $C_0$. The main result of this paper is to determine (up to finite index) the {\em monodromy group} $\rho(\pi_1({\mathcal B}^\circ))$. To state our main result, let 
\[
\Oc_o:=\Z+ \Z.2X\cong \Z +\Z\sqrt{5}
\] 
be the index $2$ subring of the ring of integers of $\Qb(\sqrt{5})$.  We will see that the monodromy representation $\rho$ factors through $ \SL_2(\Oc_o)$. In fact we will prove the following.

\begin{theorem}[{\bf Arithmeticity of the monodromy}]\label{theorem:arithmeticity}
The monodromy group of the Wiman-Edge pencil is isomorphic to a finite index subgroup of $\SL_2(\Oc_o)$; in particular it is arithmetic. 
\end{theorem}

In \S\ref{section:period} we apply Theorem~\ref{theorem:arithmeticity} to various period mappings associated to the Wiman-Edge pencil.   For example, let $\Hf$ denote the hyperbolic upper half-plane. The group $\SL_2(\Oc_0)$ acts properly discontinuously on $\Hf\times\Hf$.  The quotient of this action is a  quasi-projective, complex-algebraic surface, called a 
\emph{Hilbert modular surface}.  

The above monodromy representation $\rho$ is induced by an algebraic map 
\[
{\Bc^\circ}=\Gamma\bs\Hf\to\SL_2(\Oc_o)\bs \Hf^2,
\]
called the \emph{period map}, which assigns to a curve with faithful $\Af_5$-action its Jacobian with the induced $\Af_5$-action.  In \S\ref{subsection:period1} we use Theorem~\ref{theorem:arithmeticity} and its proof to study this period map. 

Finally, in \S\ref{subsection:K3}, we show that one can attach to the Wiman-Edge Pencil a family of K3 surfaces with base $\mathcal B$.  We then show how this description can be used to uniformize $\mathcal B$. We find:

\begin{theorem}[{\bf Uniformization of \boldmath$\Bc$}]
\label{theorem:uniformization}
The smooth, projective curve $\Bc$ (which we recall, is a copy of $\Pb^1$) supports in  a natural manner a family of polarized K3 surfaces endowed with a particular faithful action of $\mu_2\times\Af_5$ (described explicitly in \S\ref{subsection:K3}),  and the associated  period map gives $\Bc$ the structure of a Shimura curve. 
\end{theorem}

\para{Method of proof of Theorem~\ref{theorem:arithmeticity}}
As is usual with computations of monodromies, the proof of Theorem~\ref{theorem:arithmeticity} consists of two main steps.  First, in \S\ref{sect:monodromy1}, we find constraints on the monodromy in order to narrow its target to a copy of $\SL_2(\Oc_o)$; such restrictions come not only from the necessary commutation with the $\Af_5$-actions on the members of the family, but also from torsion in the Picard group of $C_0$, as well as an involutive structure coming from the extra symmetry of the Wiman curve.  The final result is to prove that in fact 
 $\rho$ takes its values in $\SL_2(\Oc_o)$.

The second step in the proof of Theorem~\ref{theorem:arithmeticity}, which we accomplish in \S\ref{sect:monodromy2}, is to prove that the image of $\rho$ has finite index. To do this, we first use Picard-Lefschetz theory to find the conjugacy classes of the local monodromies about each of the $5$ cusps of ${\mathcal B}^\circ$.  These cusps correspond to the singular members of $\Cs$: two irreducible curves, $6$-noded rational curves $C_{ir}$ and $C'_{ir}$; two curves $C_c$ and $C'_c$, each consisting of $5$ conics whose intersection graph is the complete graph on $5$ vertices; and a union $C_\infty$ of $10$ lines whose intersection graph is the Petersen graph. The group $\Sf_5$ acts on $\Cs$ with $\Af_5$ leaving each member of $\Cs$ invariant.  This action has two $\Sf_5$-invariant members: the singular curve $C_\infty$ and the Wiman curve $C_0$. The main effort of \S\ref{sect:monodromy2} is to understand these degenerations and the structures they preserve.  After improving ``up to conjugacy'' to actual elements, we are able to apply an arithmeticity criterion due to Benoist-Oh \cite{BO} to deduce Theorem~\ref{theorem:arithmeticity}.


\section{Some algebra of $\Z\Af_5$-modules}\label{sect:some algebra} 
\label{section:algebra}

We found in an earlier  paper (\cite{DFL}, Cor.\ 3.6) that the first homology group $H_1(C_o; \Cb)$ of the Wiman curve is, as a $\Cb\Sf_5$-module, twice an irreducible representation $E_\Cb$ of degree six. Since it is known that the characters of the irreducible $\Qb\Sf_5$-modules are those of the irreducible $\Cb\Sf_5$-modules,  it follows that $H_1(C_o; \Qb)$  is as a $\Qb\Sf_5$-module also twice an irreducible representation of degree six
(denoted here by $E_\Qb$). This implies that if we replace $C_o$ by an arbitrary smooth member $C$, then it is still true that $H_1(C; \Qb)\cong E_\Qb^2$ as $\Q\Af_5$-modules.

The main goal of this section and the subsequent one  is to lift this to the integral level, while also taking into account the intersection pairing. In other words, we want to identify $H_1(C_o)$ as a symplectic $\Z\Sf_5$-module. This will be used in \S\ref{sect:monodromy2}  to determine the monodromy of the Wiman-Edge pencil. The present section is only concerned with the algebraic aspects of the symplectic $\Z\Sf_5$-modules that appear here.
\\

\emph{Convention.} In this section we identify $\Af_5$ with a  triangle  group defined by  the group of motions of a regular  icosahedron. By this we mean that we make use of the following presentation of $\Af_5$:
a set of generators is
\begin{align*}
\sigma_5&=(01234), \\ 
\sigma_2&=(04)(23),\\ 
\sigma_3&=(142)
\end{align*}
and a  complete set of relations is given by prescribing their  order (indicated by the subscript) and the identity $\sigma_2\sigma_3\sigma_5=1$.
We make this more concrete in Remark \ref{rem:isocreal}.

\subsection{Irreducible $\Z\Sf_5$-modules of degree six}
Recall that the reflection representation of $\Sf_5$ is the quotient of its `natural' representation on $\Cb^5$ (given by permutation of its basis vectors) modulo the main diagonal 
$\C\hookrightarrow\C^5$ (which is a trivial representation). It is irreducible and so is the degree 6 representation  $E_\C:=\wedge^2(\C^5/\C)$. 
It is clear that this construction  is defined over $\Q$ (even over $\Z$) and so let us write  $E_\Q$ for the irreducible $\Q\Sf_5$-module 
$\wedge^2( \Q^5/\Q)$. If we consider this a  $\Q\Af_5$-module, it is still irreducible, but if we extend the scalars to $\Q(\sqrt{5})$, it will split  into two absolutely irreducible representations of dimension $3$. To be precise (we will recall and explain this below), the endomorphism
ring $K:=\End_{\Q\Af_5}E_\Q$  is isomorphic to $\Q(\sqrt{5})$ and  if we  tensor $E_\Q$ over $K$ with $\Rb$ via one of the two field embeddings 
$\sigma, \sigma': K\hookrightarrow \Rb$, we obtain real forms of  the two complex $\Af_5$-representations  of degree $3$ that differ from each 
other by an outer automorphism of  $\Af_5$  (these were denoted in \cite{DFL} by $I$ and $I'$.)

\smallskip
An obvious  integral form of $E_\Q$ is  the $\Z\Sf_5$-module $\wedge^2(\Z^5/\Z)$. 
If $\{f_i\}_{i\in \Z/5}$ is  the standard basis of $\Z^5$ and $f_{ij}$ denotes  the image  of $f_i\wedge f_j$ ($i\not=j$) in $\wedge^2(\Z^5/\Z)$, then  the set $\{ f_{ij}\}_{i\not=j}$ generates $\wedge^2(\Z^5/\Z)$ and  a complete set of  linear relations among them is  $f_{ij}=-f_{ji}$ and $\sum_j f_{ij}=0$. Note that $\{ f_{ij}\}_{i\not=j}$ is an $\Af_5$-orbit and consists of $10$ antipodal pairs. We take as our integral form the $\Z\Sf_5$-submodule $E_o$ of $\wedge^2(\Z^5/\Z)$ defined as follows. Let $\phi: \Z^5\to \Z$ be  the coordinate sum (this is a generator of  $ \Hom ( \Z^5, \Z)^{\Sf_5}$) and denote by $E_o$  the image of the $\Z \Sf_5$-homomorphism  
\[
\delta: \wedge^3(\Z^5)\xrightarrow{\iota_\phi}  \wedge^2 (\Z^5)\to \wedge^2(\Z^5/\Z),\tag{Definition $E_o$}
\]
where $\iota_\phi$ is the inner product with $\phi$ and the second map is the obvious one. In other words, $E_o$ is generated by the vectors $\delta(f_i\wedge f_j\wedge f_k)=f_{ij}+f_{jk}+f_{kj}$. The lattice  $E_o$ comes with an $\Af_5$-invariant basis, given  up to signs:


\begin{lemma}\label{lemma:generator}
Let $e:=\sum_i f_{i,i+1}\in \wedge^2(\Z^5/\Z)$. Then the $\Af_5$-orbit of  $e$ is the union of a basis of $E_o$ and its antipode.
In particular, there exists a (unique) $\Af_5$-invariant inner product 
\[
s: E_o\times E_o\to \Z
\]
for which this basis is orthonormal.
\end{lemma}
\begin{proof}
We first note that  that $e$ is fixed by the $5$-cycle $(01234)$ and that $(14)(23)$ takes $e$ to $-e$. So the $\Af_5$-orbit of $e$ consists of antipodal pairs, at most 
$60/(2.5)=6$ in number. Since $E_\Q$ is irreducible, it must be spanned by this orbit and so we have equality:  we have $6$ antipodal pairs and  the $\Af_5$-stabilizer of $e$ is generated by $(01234)$. 
It remains to show that this orbit spans $E_o$.

The identity 
$(f_{01}+f_{12}+f_{20})+(f_{02}+f_{23}+f_{30})+(f_{03}+f_{34}+f_{40})=e$ shows that $E_o$ contains $e$ and hence  the $\Z \Af_5.e$-submodule  generated  by  $e$.
On the other hand, it is straightforward to check that $e$ and its translates under $(04)(23)$  and $(124)$
sum up to $\delta(f_1\wedge f_2\wedge f_4)=f_{12}+f_{24}+f_{41}$ and since $\wedge^3 (\Z^5)$ is generated by the 
$\Af_5$-orbit of $f_1\wedge f_2\wedge f_4$, it follows that $\Z \Af_5. e$ contains $E_o$. 
\end{proof}

\begin{remark} 
The $\Af_5$-orbit of $e$ and the inner product $s$ determine each other, but  this $\Af_5$-orbit is not a $\Sf_5$-orbit, and 
so $s$ is not $\Sf_5$-invariant. Indeed, the $\Sf_5$-stabilizer of $e$ is its $\Af_5$-stabilizer (namely the cyclic group of order $5$
generated by $(01234)$) and so the $\Sf_5$-orbit of $e$ has size $24$. On the other hand, it is clear that the vectors in  $E_o$ that have unit length for $s$ make up the $\Af_5$-orbit of $e$, and so $s$ cannot be $\Sf_5$-invariant.
\end{remark}

For later use, we show that there is an equivariant map from $E_o$ to  the $\Fb_5\Sf_5$-module $N_5$ introduced in Subsection \ref{subsect:torsionpic}.
In terms of our basis, this module is the set of $\Zb$-linear combinations of  $f_0, \dots, f_4$ with coordinate sum zero, 
modulo the sublattice generated by the elements $(-5f_i+\sum_{j\in \Z/5} f_j)_{i\in \Z/5}$.
 Since $\{f_i\wedge f_j\wedge f_k\}_{0\le i<j<k\le 4}$ is a basis of $\wedge^3\Z^5$, we can define
a  homomorphism  $\tilde\psi: \wedge^3\Z^5\to N_5$ by assigning to $f_i\wedge f_j\wedge f_k$ the image of
$f_l-f_m$ in $N_5$ which is characterized by the property that  $(i,j,k,l,m)$ is an even permutation of $(0,1,2,3,4)$. 
This map is clearly onto and it is easy to see that it is also $\Sf_5$-equivariant.

\begin{lemma}\label{lemma:psi}
The homomorphism $\tilde\psi$ factors through a surjection $\psi: E_o\to N_5$ of $\Z\Sf_5$-modules.
\end{lemma}
\begin{proof}
We must show that the kernel of the map $\wedge^3(\Z^5)\xrightarrow{\iota_\phi}  \wedge^2 (\Z^5)\to \wedge^2(\Z^5/\Z)$ is contained
in the kernel of $\tilde\psi$. The kernel of the former is  generated by the  $\Sf_5$-orbit of 
$f_0\wedge f_1\wedge (f_2 +f_3+f_4)$ and $\tilde\psi(f_0\wedge f_1\wedge (f_2 +f_3+f_4))=(f_3-f_4)+(f_4-f_2)+(f_2-f_3)=0$.
\end{proof}

\subsubsection*{Some special orbits in $E_o$}
We now select an element from each antipodal pair in the $\Af_5$-orbit of $e$:
\begin{align*}
e_0:=\sigma_2(e) &=f_{41}+ f_{13}+f_{32}+f_{20}+f_{04}\\
e_i & =\sigma_5^ie_0, \; (i\in \Z/5).
\end{align*}
The icosahedral generators act on this basis as follows:
\begin{align*}
\sigma_5:&\,  e_0\mapsto e_1\mapsto e_2\mapsto e_3\mapsto e_4\mapsto e_0 \text{ (fixes $e$),}\\
\sigma_2:&\,  e\leftrightarrow e_0\,  ; \, e_1\leftrightarrow e_4\,  ; \, e_2\leftrightarrow -e_2\,  ; \,e_3\leftrightarrow -e_3 \text{ (fixes $e+e_0$),}\\
\sigma_3:&\, e\mapsto  e_0\mapsto e_1\mapsto e\,  ; \, e_2\mapsto e_4\mapsto -e_3\mapsto e_2 \text{ (fixes $e+e_0+e_1$).}
\end{align*}
This is the matrix representation of $\Af_5$ that we will use. We first note that 
the sublattice $E\subset E_0$ consisting   of integral linear combinations of our basis with even coefficient sum is $\Af_5$-invariant and of  index $2$ in $E_o$. 

The next lemma  reproduces  some of the preceding in terms of this basis:

\begin{lemma}\label{lemma:generators}
The $\Af_5$-stabilizer of $e$ is generated by $\sigma_5$. Its $\Af_5$-orbit  generates $E_o$ over $\Zb$ and consists of  the  6 antipodal pairs  $\Delta_\ir:=\{\pm e,\pm e_0, \dots ,\pm e_4\}$. 

The $\Af_5$-stabilizer of $e+e_0$ is generated by $\sigma_2$. Its $\Af_5$-orbit  generates $E$ over $\Zb$ and  consists of the 15 antipodal pairs
$\Delta_\infty:=\{\pm(e+e_i), \pm(e_i+e_{i+1}), \pm(e_{i+1}-e_{i+1})\}_i$. 

The $\Af_5$-stabilizer of $e+e_0+e_1$ is generated by 
 $\sigma_3$. Its $\Sf_5$-orbit equals its $\Af_5$-orbit,  generates $E_o$ over $\Zb$ and consists of  the 10 antipodal pairs 
$\Delta_c:=\{\pm(e+e_{i}+e_{i+1}),\pm(e_i -e_{i-2}-e_{i+2})\}_i$. 

The lattice $E$ is $\Sf_5$-invariant.
\end{lemma}
\begin{proof}
We already established the first assertion.

Since $\sigma_2$ stabilizes $e+e_0$, its orbit has at most $30$ elements. That  it contains the 15 pairs listed is straightforward to verify (for example, $\sigma_5^i(e+e_0)=e+e_i$ and then note that for $i=0,1,2,3,4$,  the vector $\sigma_2(e+e_i)$ equals
resp.\  $e_0+e$, $e_0+e_4$, $e_0-e_2$, $e_0-e_3$, $e_0+e_1$). This orbit is contained in $E$ and the  subset  $(e_1-e_2,e_2-e_3, e_3-e_4, e_4-e_0,e_0-e,e_0+e)$ of this orbit is a basis of  $E$.

We next consider the orbit of $e+e_0+e_1$. We compute
\[
e+e_0+e_1=f_{12}+f_{24}+ f_{41} 
\]
and this shows that $e+e_0+e_1$ is not only stabilized by $\sigma_3=(142)$, but also by the transposition  $(03)$. This implies that its  $\Sf_5$-orbit of
$e+e_0+e_1$ equals its $\Af_5$-orbit. This  orbit has at most $20$ elements and we show
that this orbit contains the $10$ pairs listed. We have $\sigma_2\sigma_5(e+e_0+e_1)=\sigma_2(e+e_1+e_2)=e_0+e_4-e_2$ and the $\sigma_5$-orbits of  $e+e_0+e_1$ and $e_0+e_4-e_2$ yield all the listed pairs up to sign. Since $\sigma_5^2\sigma_2\sigma_5^{-2}(e_0+e_4-e_2)=\sigma_5^2\sigma_2(e_3+e_2-e_0)=\sigma_5^2(-e_3-e_2-e)=-e_1-e_0-e$, it is also invariant under taking the opposite. This orbit 
generates $E_o$:  it contains  $\sigma_2(e+e_0+e_1)-(e+e_0+e_1)=e_4-e_1$ and with it then the span of the $\Af_5$-orbit of $e_4-e_1$, that is $E$.
Since $e+e_0+e_1\notin E$ and $E$  has index $2$ in $E_o$, we get all of $E_o$. 

As to the last statement, we have seen in the proof of Lemma \ref{lemma:generator} that $\delta(f_1\wedge f_2\wedge f_4)=e +e_0+e_1$. Since the 
$\Af_5$-orbit of the latter generates $E_o$, it follows that $E$ can also be characterized as the set of $\Z$-linear combinations of the 
$\delta(f_i\wedge f_j\wedge f_k)$ with even coefficient sum. This lattice is clearly $\Sf_5$-invariant.
\end{proof}

The $s$-dual of $E$, denoted $E^\vee$, consists by definition of the $e\in E_\Q$ with $s(e, e')\in\Z$ for all 
$e'\in E_o$. It contains $E_o$ as a sublattice of index $2$ and  a representative of the nontrivial coset is $\varepsilon:=\tfrac{1}{2}(e +\sum_{i\in \Z/5} e_i)$. We have $E^\vee/E\cong \Z/2\oplus \Z/2$  with the nonzero elements being represented by $\varepsilon$, $e$ and $\varepsilon+e$. The action of
$\Af_5$ on $E^\vee/E$ is trivial; this can be verified by computation, but this also follows from the fact that $\Af_5$ is simple so that any 
action of  $\Af_5$ on a $3$-element set must be trivial.

\begin{remark}
This situation is familiar in the theory of root systems:  the $\alpha\in E$ with $s(\alpha, \alpha)=2$  make up a root system of type  $D_6$  
that generates $E$ (so $E$ is the root lattice) and  $E^\vee$ the weight lattice. The fact that $\Af_5$ acts trivially on $E^\vee/E$ implies that $\Af_5$ embeds in the Weyl group of this root system.) 
\end{remark}

\subsection{Commutants of $\Z\Af_5$-modules}
We shall see that the $\Z\Af_5$-modules above admit endomorphisms that are nontrivial in the sense that they are not multiples of the identity.
One such element is $X\in \End(E^\vee)$, defined by 
\begin{align*}
X(\varepsilon):= & \varepsilon +e,\\ 
X(e_i):=& \varepsilon-(e_{i+2}+e_{i-2}).
\end{align*}

\begin{lemma}
The endomorphism  $X$  is selfadjoint  with respect to $s$ and satisfies $X^2=X+1$.  In particular, $X$  preserves $E$ and $E^\vee$ so that $X$ also acts on $E^\vee/E$.
We have $X(e)= \varepsilon$ and (so) $X$ acts transitively on the set  of the (3) nonzero elements of $E^\vee/E$. 
 Moreover, $X$ commutes with the $\Af_5$-action.
\end{lemma}
\begin{proof}
We  only verify the last assertion, as checking the others is straightforward.
Since $\sigma_5$ and $\sigma_2$ generate $\Af_5$, it suffices to  check that these elements commute with $X$. This is obvious for $\sigma_5$.
In the case of $\sigma_2$, we must verify that 
\[
\sigma_2:   X(e)\leftrightarrow X(e_0)\,  ; \, X(e_1)\leftrightarrow X(e_4)\,  ; \, X(e_2)\leftrightarrow -X(e_2)\,  ; \,X(e_3)\leftrightarrow -X(e_3).
\]
This is also straightforward. 
\end{proof}

Thus $E$ becomes a module over the ring $\Oc:=\Z[X]/(X^2-X-1)$. Notice that the map $X\mapsto\half +\half\sqrt{5}$ identifies  $K= \Q[X]/(X^2-X-1)$ with the  number field $\Q(\sqrt{5})$ and $\Oc$ with its ring of integers.  Any unit of $\Oc$ is an integral power of $X$ up to sign.
It is clear that $E$ is a torsion free $\Oc$-module of rank $3$ and $E_\Q$ a
$K$-vector space of dimension $3$. Since $K$ has class number $1$, $E$ is in fact a free $\Oc$-module.
For the same reason this is true for  $E^\vee$. Since $X$ acts transitively on the  nonzero elements of $E^\vee/E$, there are no intermediate  $\Oc$-submodules $E\subsetneq L\subsetneq  E^\vee$. We note that the  $K$-stabilizer of $E_o$ in $E_\Q$ is the subring
\[
\Oc_o:=\Z+ \Z.2X\cong \Z +\Z\sqrt{5}
\] 
of $\Oc$ of index $2$. The group of units of $\Oc_o$ is generated by $-1$ and $X^3=2X+1$;  it contains the subgroup of totally positive units of $\Oc_o$ as a subgroup of index $2$, and the latter is generated by $X^6=8X+5$. 

\begin{remark}[{\bf The icosahedral realizations}]\label{rem:isocreal}
When we regard $E_\Q$ as a $K\Af_5$-module of degree $3$, it is absolutely irreducible. For example,  
we have two field embeddings $\sigma, \sigma' : K\hookrightarrow \Rb$ characterized by
$\sigma (X)=\tfrac{1}{2}(1+\sqrt{5})$ resp.\ $\sigma' (X)=\tfrac{1}{2}(1-\sqrt{5})$ which are exchanged by the nontrivial Galois involution of $K$ and 
the associated $\Rb\Af_5$-modules  $\Rb\otimes_{K,\sigma} E_\Q$ and $\Rb\otimes_{K,\sigma'} E_\Q$ are irreducible. Since the  $K$-action on $E_\Q$  is self-adjoint with respect to $s$,  the inner product extends to a symmetric $K$-bilinear form 
$s_K: E_\Q\times E_\Q\to K$ and the two field embeddings define  $\Af_5$-invariant inner products on $I_\Rb$ and $I'_\Rb$  preserved 
by the $\Af_5$-action. The convex hull of the $\Af_5$-orbit
of the image of $e$ in each of these  is a regular icosahedron relative to this inner product, thus making explicit the realization of $\Af_5$ as  the icosahedral group.
\end{remark}

\begin{lemma}
The commutant of the $\Z\Af_5$-module $E$ resp.\ $E_o$ is $\Oc$ resp.\ $\Oc_o$. 
Conjugation with an  element of $\Sf_5\ssm \Af_5$ induces in these rings the Galois involution (which sends $X$ to $1-X$).
\end{lemma}
\begin{proof}
Since $E_K$ is absolutely irreducible as $\Af_5$-representation, $\End_{K\Af_5}(E_K)=K$  by Schur's lemma. So $\End_{\Z\Af_5}(E)$ is a subring of $K$. The integrality implies that this subring must be contained in $\Oc$. On the other hand, the previous lemma shows that it contains  $\Oc$ so that we have equality. The proof that $\End_{\Z\Af_5}(E_o)=\Oc_o$ is similar.

We also know that $E_\C$ is irreducible as an $\C\Sf_5$-module, and so Schur's lemma implies that the commutant of the $\Z\Sf_5$-module $E$ is just 
$\Z$. Hence conjugation with an  element of $\Sf_5\ssm \Af_5$ induces a nontrivial involution of the ring $\Oc$ with fixed point ring $\Z$. 
There is only such involution, namely   the Galois involution of $\Oc$.
\end{proof}

It is clear that  $E^\vee$ is also $\Oc$-invariant. The definition of $X$ shows that the action of $\Oc$ on $E^\vee/E$  factors through a faithful action  of $\Oc/2\Oc$. 
But $2\Oc$ is a prime ideal of $\Oc$ so that the finite ring $\Oc/2\Oc$ is a field with $4$ elements (hence denoted $\Fb_4$). It has  the order 2 subring $\Oc_o/2\Oc$ 
as its prime field $\Fb_2\subset \Fb_4$.  Thus $E^\vee/E$ acquires  the structure of a $1$-dimensional vector space over $\Oc/2\Oc=\Fb_4$. 
The subgroup $E_o/E\subset E^\vee/E$ is a module over $\Oc_o/2\Oc=\Fb_2$,  and so  defines an $\Fb_2$-form of the $\Fb_4$-line $E^\vee/E$.

\begin{remark}\label{rem:}
One may check that the $\Sf_5$-orbit of $e$ is the union of two $\Af_5$-orbits, namely  of $e$ and of $(2X+1)e$. Since the 
latter is a  vector of $s$-length $3$, this makes it evident that $s$ is \emph{not} preserved by $\Sf_5$. Nevertheless, since  $2X$ takes $E^\vee$ to $E$, $\Sf_5$ will preserve each coset of $E$ in $E^\vee$ (in other words, will act  as the identity in $E^\vee/E$).
\end{remark}

\subsection{The functors $V$ and $V_o$}
 Let $H$ be a  finitely generated $\Z\Af_5$-module. Then  
the isogeny  module 
\[
V_o(H):=\Hom_{\Z\Af_5} (E_o, H)\; \text{ resp.\;  } V(H):=\Hom_{\Z\Af_5} (E, H)
\]
is in a natural manner an $\Oc_o$-module resp.\  $\Oc$-module (acting by precomposition). So $V_o$ resp.\ $V$ is a  functor from the category of  finitely generated $\Z\Af_5$-modules to the category of  finitely generated $\Oc_o$-modules resp.\  $\Oc$-modules. 
Restriction defines a natural transformation  $V_o\to V$. The evaluation map $V_o(H)\times E_o\to  H$ factors through a homomorphism $V_o(H)\otimes_{\Oc_o} E\to  H$
of $\Z\Af_5$-modules. 
There will be two cases of special interest to us. 
\\

First assume  that $H$ is free as a $\Z$-module. Then both isogeny modules are torsion free, but in the case of $V_o(H)$, it need not be free. In fact, $V_o$ applied to the chain $E\subset E_o\subset E^\vee$  yields the chain of $\Oc_o$-modules  ${\Oc}\xrightarrow{\times 2} \Oc_o\subset \Oc$, and $\Oc$ is not free as a $\Oc_o$-module. On the other hand, $V(H)$ is a free $\Oc$-module, as $\Oc$ has class number $1$. (Indeed, if we apply $V$ to the above chain we find that $E\subset E_o$ induces an isomorphism $V(E)=V(E_o)=\Oc$ and that $E_o\subset E^\vee$ induces 
$V(E_o)=\Oc\xrightarrow{\times 2}\Oc= V(E^\vee)$.)

Suppose now $H$ is also endowed  with a $\Af_5$-invariant symplectic form $(x,y)\in H \times H \mapsto x\cdot y\in  \Z$. Then for every pair $v_1, v_2\in V_o(C)$, the form
\[
(x,y)\in E_o\times E_o\mapsto   v_1 (x)\cdot v_2(y)\in \Z
\]
is also  $\Af_5$-invariant. This means that there exists a unique $\Af_5$-equivariant endomorphism  $A(v_1, v_2)$ of $E_o$ such that 
\[
v_1 (x)\cdot v_2(y)= s(A(v_1, v_2)(x), y)
\] 
for all $x,y\in E$. But any such endomorphism is in $\Oc_o$. Using the fact that $X$ is self-adjoint with respect to $s$,  one checks that the resulting map $A: V_o(C)\times V_o(C)\to \Oc_o$ is symplectic: it is $\Oc_o$-bilinear, antisymmetric and becomes nondegenerate over $K$. 
\\

The other case is of interest is when $H$ is the $\Fb_5\Af_5$-module $N_5$ defined in Subsection \ref{subsect:torsionpic}. Recall that we  constructed in Lemma \ref{lemma:psi} a surjection $\psi: E_o\to N_5$. So this is a nontrivial element of $V_o(N_5)$.
As $N_5$ is an $\Fb_5$-vector space, so will be $V_o(N_5)$. At the same time it is an $\Oc_o$-module. Indeed, the prime $5$ ramifies in $\Oc_o$, for  $\Oc_o$ is additively  generated by $1$ and $2X-1$ and  $2X-1$  
is a square root of $5$. So  $2X-1$ generates a prime ideal $\Oc_o$ with  residue field $\Fb_5$ and  the $\Oc_o$-module structure on $V_o(N_5)$ factors through this residue field. 

\begin{lemma}\label{lemma:toN5}
The $\Oc_o$-module $V_o(N_5)$ is a vector space over $\Fb_5$ of dimension one, generated by $\psi$.
\end{lemma}
\begin{proof}
Let $\psi'\in V_o(N_5)$. Since $f_{01}+f_{12}+f_{20}$ generates $E_o$ as a $\Af_5$-module, it suffices to prove that $\psi'(f_{01}+f_{12}+f_{20})$ is  
unique up to a scalar in $\Fb_5$. Let  us represent $\psi'(f_{01}+f_{12}+f_{20})$  by $\sum_i a_if_i$ with $a_i\in\Z$ such that $\sum_i a_i=0$. 
If we sum over the orbit of $(012)$, we find that
 $3\psi'(f_{01}+f_{12}+f_{20})$ is represented by an element of the form $a (f_0+f_1+f_2)+bf_3+cf_4$ with $3a+b+c=0$. This element is also represented by 
 \[a (f_0+f_1+f_2)+bf_3+cf_4 -a(f_0+f_1+f_2+f_3-4f_4)=(b-a)f_3+(c+4a)f_4=(b-a)(f_3-f_4).\]  This proves that $\psi'(f_{01}+f_{12}+f_{20})$ is unique up to scalar.
\end{proof}

\section{Structures preserved by the monodromy}
\label{sect:monodromy1}

The Wiman-Edge pencil ${\Bc^\circ}$ is a family of genus $6$ smooth algebraic curves. Since the action of 
$\pi_1({\Bc^\circ})$ on the integral first homology of the fiber preserving algebraic intersection number, we obtain a {\em monodromy representation} $\pi_1({\Bc^\circ})\to \Sp_{12}(\Zb)$.  The monodromy action preserves a lot more structure, for example it intertwines the $\Af_5$ automorphism group of each fiber.   Our goal in this section is to find and describe other structure preserved by the monodromy, thus giving strong restrictions on its image.  

\subsection{Torsion in the Picard group of the Wiman-Edge pencil}
\label{subsect:torsionpic}
Recall that the Wiman-Edge pencil (when regarded as lying  on the del Pezzo surface $S$) has as its base  locus
the  irregular $\Sf_5$-orbit  $\Sigma$ in $S$ of size 20. 
We follow \cite{DFL} and denote  by $H_0^2(S)\subset \pic (C)$ the orthogonal complement of the anticanonical class. 
This is a negative definite lattice spanned by its elements of self-intersection $-2$. Each such $(-2)$-vector can be represented by  the 
difference  of two disjoint lines and together they make up a root system of type $A_4$.  

For the discussion in this subsection, we shall regard $S$ as obtained from $\Pb^2$ by blowing up in $4$ points in general position. Then a basis of  $H_0^2(S)$ is 
$(\ell, \varepsilon_0,\varepsilon_1, \varepsilon_2, \varepsilon_3)$, where 
the $\varepsilon_i$'s  are the classes of the exceptional curves and $\ell$ is the image of a class of a line in $\Pb^2$ (see \S 3 of \cite{DFL}, where the notation slightly differs from the one used there). The anti-canonical class of $S$ is  $-3\ell+\varepsilon_0+\varepsilon_1+\varepsilon_2 +\varepsilon_3$,  a root basis of its orthogonal complement $H^2_0(S)$ is  $(\alpha_1, \alpha_2, \alpha_3, \alpha_4)=(\varepsilon_0-\varepsilon_1, \varepsilon_1-\varepsilon_2, \varepsilon_2-\varepsilon_3, \ell-\varepsilon_1-\varepsilon_2-\varepsilon_3)$ and the   $10$ line classes of $S$ are $\{\varepsilon_i\}_{i=0}^3$ and $\{\ell-\varepsilon_i-\varepsilon_j\}_{0\le i<j\le 3}$. The $\Sf_5$-action is realized as the Weyl group of this root system and we choose an identification which  
makes $\Sf_4$ correspond with the stabilizer of $\ell$, i.e., the full symmetric group of $\{\varepsilon_i\}_{i=0}^3$.

\subsubsection*{A finite group associated with the root system in $H_2(S)$} The intersection pairing identifies the dual lattice $H^2_0(S)^\vee=\Hom (H^2_0(S), \Z)$ with a sublattice of $H^2_0(S; \Q)$ of vectors that have integral intersection product with vectors in $H^2_0(S)$. 
This is  the weight lattice of the above root system; it contains  $H^2_0(S)$.
It is known that  $H^2_0(S)^\vee/ H^2_0(S)$ is cyclic of order $5$,  with  a generator representable by  the fundamental weight $\varpi_4\in H^2_0(S)^\vee$ 
defined by  $\varpi_4\cdot\alpha_i=-\delta_{i4}$ (this is also the orthogonal projection of $-\ell$ in $H^2_0(S); \Q)$). The orbit of $\varpi_4$ under the Weyl group generates $H^2_0(S)^\vee$. Since $5\varpi_4=\alpha_1+2\alpha_2+3\alpha_3+4\alpha_4$ is indivisible in $H^2_0(S)$, 
\[
N_5:= H^2_0(S)/5H^2_0(S)^\vee
\]
is an $\Fb_5$-vector space of dimension $4-1=3$. Essentially by construction, the intersection pairing induces a nonsingular quadratic form 
$N_5\times N_5\to \Fb_5$. Note that $N_5$ is an  $\Fb_5\Sf_5$-module on which  $\Sf_5$ acts orthogonally (it is an incarnation of the standard representation
$\SL_2(\Fb_5)$ of degree $3$, which indeed leaves invariant a quadratic form). 
In terms of the usual description of the root system $A_4$ (see for example \cite{bourbaki:lie}), $H^2_0(S)$ is identified with the corank one sublattice of $\Z^5$ consisting of vectors with coordinate sum zero and $5H^2_0(S)^\vee$ with the
sublattice generated by the vectors with four  coordinates equal to $-1$ and the remaining coordinate  equal to $4$.
The $\Sf_5$-action is the obvious one  and from this we easily see that  every  $\Sf_5$-invariant element of $N_5$  and every  $\Sf_5$-invariant $\Fb_5$-valued linear form on $N_5$ is zero. Since $\Fb_5\Af_5$ has only the trivial representation in dimension one, this implies that   $N_5$ is irreducible as an  $\Fb_5\Af_5$-module. 

\subsubsection*{A map to the Jacobian}
Let $C$ be a member of the Wiman-Edge pencil. Every point of $\Sigma$ lies in the smooth part of $C$ so that it defines a Cartier divisor of degree one on  $C$. We thus obtain a homomorphism
$H_0(\Sigma)\to \pic (C)$. This map restricts to a homomorphism $\tilde H_0(\Sigma)\to\pic_0(C)$, where the source is reduced homology and the target is the degree zero part of $\pic (C)$ (which is also the Jacobian of $C$, when $C$ is smooth).
A line in $S$ meets $C$ in $\Sigma$ in an opposite pair in $\Sigma$ whose sum
represents the image of this line under the restriction map $H^2(S)=\pic (S)\to \pic (C)$. 
So if $\tilde H_0(\Sigma)^+\subset\tilde H_0(\Sigma)$ stands for the sublattice for which opposite pairs have the same coefficient, then 
the homomorphism $\tilde H_0(\Sigma)^+ \to\pic_0(C)$ factors through a surjection $\tilde H_0(\Sigma)^+ \to H^2_0(S)$. 
Observe that all these maps are $\Af_5$-equivariant.

\begin{proposition}\label{prop:5torsion}
When $C$ is smooth, the  kernel of the restriction map $H^2_0(S)\to\pic_0(C)$  factors through an $\Af_5$-equivariant embedding of
$N_5$ in the $5$-torsion of $\pic_0(C)$.
\end{proposition}

The proof rests on the fact that if we  restrict the $\Af_5$-action to Klein's Vierergruppe $\Vf_4\subset \Sf_4$ (the abelian subgroup of $\Af_4$ whose nontrivial elements are the three elements of order 2),  then the orbit space is still of genus zero:

\begin{lemma}
Let  $C\to \Vf_4\backslash C:=\overline C$ form the orbit space. Then $\overline C$ is rational and the degree $4$ cover $C\to \overline C$  has nine singular fibers in which we have simple ramification.
\end{lemma}
\begin{proof}
In order to apply Riemann-Hurwitz, we need to determine the ramification data.  Since $H_1(C; \C)$ is as a $\Af_5$-representation of type $2I+2I'$, the trace of every order 2  element of $\Sf_5$ on $H_1(C; \C)$ is $-4$ and so the Lefschetz number of such an element is $1-(-4)+1=6$. 
Hence it has as many fixed points.  The stabilizer of a point of $C$ is cyclic and so the fixed point sets of  distinct order 2  elements are disjoint. We therefore find that
$C\to \overline C$ has $6.3=18$ points of simple ramification. Since $C\to \overline C$ is a Galois cover of degree $4$, it follows that  we have $18/2=9$ singular fibres. 
The identity of Euler numbers $4e(\overline C)-18=e(C)=-10$ shows that  $e(\overline C)=2$. This proves that $\overline C$ is rational. 
\end{proof}

Since  $\pic_0$ of a rational curve is trivial, the fibers of $C\to \overline C$ all define the same class in $\pic(C)$. We 
will use the preceding lemma via this  implication.

\begin{proof}[Proof of Proposition \ref{prop:5torsion}]
Every element  of $\Sigma$ has an $\Af_5$-stabilizer of order $3$ and so is not a ramification  point of 
the projection  $C\to \overline C$. It follows that $\Sigma$ is the union of $60/(4\cdot 3)=5$ regular fibers of  $C\to \overline C$.
We want to understand how the antipodal involution of $\Sigma$ acts in these fibers. 

Since we agreed that the subgroup $\Sf_4$ of $\Sf_5$  is realized as the permutation group of $\{\varepsilon_i\}_i$, the Vierergruppe  $\Vf_4$ becomes a subgroup of this permutation group. We see that the $\Vf_4$-orbits in the set of line classes are the $4$-element set $\{\varepsilon_i\}_i$
and the three $2$-element sets of the form $\{\ell-\varepsilon_i-\varepsilon_j, \ell-\varepsilon_k-\varepsilon_l\}$, where $i,j, k, l$ are mutually distinct. 
The image of $\sum_i \varepsilon_i$  in $\pic(C)$ is the represented by the sum of  $4$ antipodal pairs and since $\Vf_4$ permutes the $\varepsilon_i$'s transitively,
it follows that  they are also the sum of two regular fibers of $C\to \overline C$. By the same reasoning,  
the image of the $\Vf_4$-invariant element $(\ell-\varepsilon_0-\varepsilon_1)+ (\ell-\varepsilon_2-\varepsilon_3)$ is represented $2$ antipodal pairs in a single fiber.
These fibers are linearly equivalent and so the kernel of $\pic (S)\to \pic (C)$ contains $n:=-\sum_i \varepsilon_i+2(2\ell-\varepsilon_0-\varepsilon_1-\varepsilon_2-\varepsilon_3)=4\ell-3\sum_i \varepsilon_i$.
This  is a sum  of roots:  $n=\sum_{0\le i<j<k\le 3} (\ell-\varepsilon_i-\varepsilon_j-\varepsilon_k)=\alpha_1+2\alpha_2+3\alpha_3+4\alpha_4$. It is in fact equal to $5\varpi_4$.
Since $n$ is fixed by a reflection, its $\Af_5$-orbit is the full $\Sf_5$-orbit. The lattice generated by this orbit is $5H^2_0(S)^\vee$ and as  
the kernel must be invariant under $\Af_5$,  it will contain this sublattice. This proves the factorization. 

So restriction defines a homomorphism from $N_5\to \pic (C)[5]$. This   map
is $\Af_5$-equivariant and hence its kernel is a  $\Af_5$-invariant subspace of $N_5$. Since $N_5$ is irreducible, this kernel  is either trivial or  all of 
$H^2_0(S)/5H^2_0(S)^\vee$. Suppose the latter. This then means that for every antipodal pair in $C$ its sum is a degree $2$ divisor whose class
is independent of that pair. The linear system of this class  then defines a pencil of degree $2$ on $C$ for which the antipodal pairs are fibers. A 
degree 2 pencil  on a curve of positive genus must define a hyperelliptic involution and hence is intrinsic to the curve. In our case, this hyperelliptic
involution must be normalized by the $A_5$-action. But the automorphism group of $C$ is contained in $\Sf_5$  (with equality for the  Wiman curve) and 
no element of  $\Sf_5$ normalizes $\Af_5$. As this yields a contradiction, this proves that $N_5\to \pic (C)[5]$ has trivial kernel.
\end{proof}

\begin{remark}\label{rem:cansurjection}
Since $\pic_0(C)[5]$ can be identified with $\Hom (H_1(C), \mu_5)$,  we have an embedding of $N_5$ in $\Hom (H_1(C), \mu_5)$. Dually, this  yields an surjection $H_1(C)\twoheadrightarrow\Hom (N_5, \mu_5)$. If we choose a primitive $5$th root of unity (which identifies $\mu_5$ with
$\Fb_5$) and use the quadratic form to identify $N_5$ with its dual, then we obtain an isomorphism  $\Hom (N_5, \mu_5)\cong N_5$ and there results
a surjection $H_1(C)\twoheadrightarrow N_5$. 
Proposition \ref{prop:5torsion} has therefore a topological consequence: the surjection $H_1(C)\twoheadrightarrow N_5$ is locally constant  as $C$ varies in the smooth fibers of the Wiman-Edge pencil, and so this imposes restriction (a `level structure') on the monodromy 
of this family. We shall see this illustrated when we compute the monodromy group in \S \ref{sect:monodromy1} and \S\ref{sect:monodromy2}.
\end{remark}

\begin{remark}
In the above argument we divided out by the Klein Vierergruppe. If we divide out by the bigger group $\Af_4$ instead, then the orbit space  \emph{a fortiori}
of genus zero. This intermediate orbit space defines a (non-Galois) cover of the $\Af_5$-orbit space $\tilde P\to P$ of degree $5$.  Such a covering can be given  by a rational function $f$ on the smooth rational curve  $\tilde P$ of degree $5$. The monodromy group of $f$ is $\Af_5$ so that the Galois closure of the associated degree $5$ extension of rational function fields $C(\tilde P)/\C (f)$ defines an $\Af_5$-covering. This covering will be a copy of  $C\to P$.
\end{remark}

If $C$ is a smooth member $C$ of the Wiman-Edge pencil then $H_1(C)$ is a symplectic $\Z \Af_5$-module,  and we then abbreviate 
$V_o(H_1(C))$ by $V_o(C)$. Recall that in Remark \ref{rem:cansurjection} we found (after identifying $\mu_5$ with $\Fb_5$) a natural surjection
$H_1(C)\to N_5$ of $\Z \Af_5$-modules. The following is then immediate from Lemma \ref{lemma:toN5}:

\begin{corollary}\label{cor:toN5}
There is a natural surjection $V_o(C)\to \Fb_5$ that is locally constant when $C$ varies in the smooth members of the Wiman-Edge pencil. 
\end{corollary}

\subsection{The homology of the Wiman curve  as a symplectic $\Z\Sf_5$-module}\label{subsect:Wimanhomology}

The goal of this subsection is to determine $H_1(C_o)$  as a symplectic $\Z\Sf_5$-module.  This information will help us to determine the global monodromy group of the Wiman-Edge pencil.

Recall that  the $\Sf_5$-action on the Wiman curve $C_o$ makes it a $\Sf_5$-orbifold cover of an orbifold $P_o$  of type $(0; 6,4,2)$ and that the restriction of the action to $\Af_5$ defines an intermediate orbifold $P$ of type $(0; 3,2,2,2)$ such that  
$ P\to P_o$ is of degree 2 and ramifies over the orbifold points of orders $6$ and $4$. In order to identify $H_1(C_o)$ as a symplectic 
$\Z\Sf_5$-module, we take a closer look at this situation. 

Let $w$ be the affine coordinate on $P_o$ such that the orbifold points  of order $6$,  $4$ and $2$ are given by $w=\infty$, $w=0$ and 
$w=1$ respectively (this makes a $P_o$ as a projective line defined over $\Rb$ (even over $\Qb$). 
Then $P$  is also defined over $\Rb$: it admits an affine coordinate $z$ for which $w=z^2$, so that the  orbifold point   of order $3$ is given by $z=\infty$,  and the orbifold points  of order $2$ are $z=0, 1, -1$. Note that $z$ is unique up to sign. 

The preimage of the real projective line (that is, where $w$ is real) in $C_o$ defines a $\Sf_5$-invariant triangulation of 
$C_o$; if we endow $C_o$ with its hyperbolic structure, then this is in fact a hyperbolic triangulation. Let $K\subset C_o$ be a $2$-simplex of this 
triangulation which maps onto the upper half plane of $P_o$. We denote the vertices of $K$ by $p_6$, $p_4$, $p_2$ according to the
order of their stabilizer. The stabilizer of such a point is cyclic and the orientation of $C_o$ singles out  a natural generator $\tau_j$ (counter clockwise rotation over $2\pi/j$ around $p_j$).
It is elementary to see that the cycle type of these generators is  
$(3,2)$ for $\tau_6$, $(4)$ for $\tau_4$  (so both are odd) and $(2,2)$ for $\tau_2$ (so $\tau_2$ is even) and that $\tau_6\tau_4\tau_2=1$; see \S 2.3 of \cite{FL}.  In fact, 
Theorem 2.1 of \cite{FL} implies that any ordered triple  $(\tau_6, \tau_4, \tau_2)$ of generators  $\Sf_5$ whose orders are as their subscript and  satisfy $\tau_6\tau_4\tau_2=1$,  differ from the triple above by an inner automorphism. So any such triple comes from some choice of  $K$. We shall exploit this below. 

Let $K^*\subset C_o$ be the  geodesic $2$-simplex adjacent to $K$  that has in common with $K$ the edge $p_4p_6$.
Then $K\cup K^*$ is a fundamental domain of the $\Sf_5$-action on $C_o$.  We denote the vertex of $K^*$ distinct from $p_4$ and $p_6$ by $p_2'$. So $p'_2=\tau_4^{-1}p_2$ and its $\Sf_5$-stabilizer is generated
by $\tau'_2:=\tau_4^{-1}\tau_2\tau_4$. Every edge of $K$ or $K^*$ lies on a closed geodesic that lies over an interval  in $P_o(\Rb)$. 

We write $\alpha$ resp.\  $\alpha'$ for the complete geodesic in $C_o$ that contains the geodesic segment $[p_4,p_2]$ resp.\ $[p_4,p'_2]$. Both map to the segment
$[0,1]$ in $P_o(\Rb)$, but their images in $P$ are distinct and consist of  $[0,1]$ resp.\  $[0, -1]$. 
It is clear that both $\tau_4^2$ and $\tau_2$ leave $\alpha$ invariant and act in $\alpha$  as a reflection. They generate in $\alpha$ a reflection group with  $[p_4,p_2]$ as fundamental domain. Since  $[p_4,p_2]$ maps injectively to the $\Sf_5$-orbit space, it follows that
this  subgroup of $\Sf_5$ is in fact the $\Sf_5$-stabilizer of $\alpha$. Both generators 
are even permutations, and so this stabilizer is in fact a subgroup of $\Af_5$. The  product 
$\tau_4^2\tau_2$ is easily shown to be of order $3$ (otherwise, see our specific choice for the $\tau_i$'s below), and so this stabilizer is a dihedral reflection group of order $6$.  It follows that $\alpha$ has $120/6=20$ 
$\Sf_5$-translates and $60/6=10$  $\Af_5$-translates. Since $\alpha'$ is a translate of $\alpha$ under an odd permutation (namely $\tau_4$), it cannot be a 
$\Af_5$-translate. Note that if  $\vec\alpha$ stands for $\alpha$ with the orientation  defined by $[p_4,p_2]$, then the 
stabilizer of $\vec\alpha$ is the cyclic group generated by $\tau_4^2\tau_2$ and hence its $\Af_5$-obit consists of $10$ oppositely oriented  pairs. 

\begin{figure}
\centerline{\includegraphics[scale=0.300]{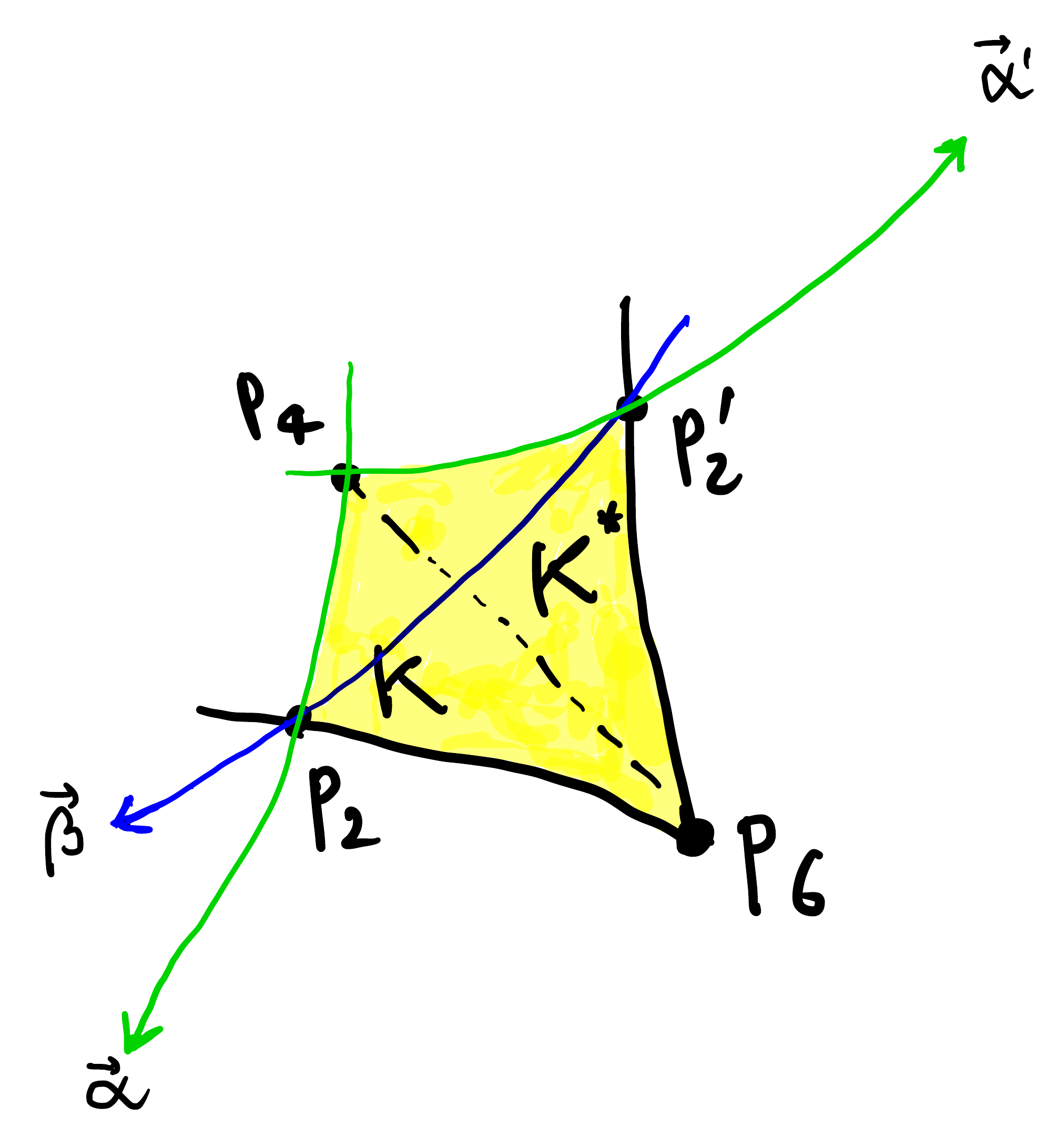}}
\caption{\footnotesize The closed oriented geodesics $\vec\alpha$, $\vec\alpha'$ and $\vec\beta$ on the Wiman curve.}
\end{figure}

We will also be interested in the geodesic $\beta$ on $C_o$ that contains  the geodesic segment $[p'_2,p_2]$. This geodesic is also closed; it
maps in $P_o$ to the unit circle $|w|=1$  and hence (upon perhaps replacing $z$ by $-z$) its image in $P$ will be the semicircle $|z|=1$, $\Re (z)\ge 0$. 
A similar argument
shows that the stabilizer of $\beta$ is the reflection group generated by the even permutations $\tau_2$ and $\tau'_2$. The product
$\tau_2\tau_2'$ has order $5$ and hence the stabilizer of $\beta$ is a dihedral subgroup $\Af_5$ of order $10$, whereas the stabilizer of 
$\vec\beta$ (the orientation being given by $[p'_2, p_2]$) is generated by $\tau_2\tau_2'$.  It follows that 
$\beta$ has $120/10=12$  $\Sf_5$-translates and $60/10=6$  $\Af_5$-translates. A $\Sf_5$-translate which is not an 
$\Af_5$-translate is for instance $\beta':=\tau_4(\beta)$. Its image in $P$ will be the semicircle $|z|=1$, $\Re (z)\le 0$.

Since the  $\Af_5$-orbit of $\alpha$ resp.\  $\beta$ is the preimage of an arc in $P$ that connects two orbifold points of order $2$,  
it must consist of resp.\  $10$, $6$ closed geodesics \emph{that  are pairwise disjoint}.  As shown in \S 2 of \cite{FL}, these $\Af_5$-orbits make up a configuration of $K_5$-type resp.\ dodecahedral type. The same is true for the  $\Af_5$-orbits of $\alpha'$ and $\beta'$.
Recall that at the beginning of \S\ref{sect:some algebra}  we specified the generators $\sigma_5=(01234)$, $\sigma_3=(142)$ and $\sigma_2=(04)(23)$ for $\Af_5$.

\begin{lemma}\label{lemma:goodchoice}
We can  choose $K$ such that the associated triple  $(\tau_6, \tau_4, \tau_2)$ in $\Sf_5$ has the property that $\sigma_3$ resp.\ $\sigma_5$ generates the stabilizer of $\vec\alpha$ resp.\ $\vec\beta$ and $\sigma_5^{-1}\sigma_3\sigma_5$ stabilizes $\vec\alpha':=\tau_4(\vec\alpha)$.
\end{lemma}
\begin{proof}
We take 
\[
\tau_6=(012)(34),\quad  \tau_4=(0432),\quad  \tau_2=(03)(12).
\] 
Then $\tau_6 \tau_4 \tau_2=1$.  Since $\tau_4^2=(03)(24)$, $\alpha$ is stabilized by $\tau_4^2\tau_2=(03)(24)(03)(12)=(24)(12)=(142)=\sigma_3$.

Furthermore,  $\tau'_2=\tau_4\tau_2\tau_4^{-1}= (0432)(03)(12)(0234)=(01)(24)$ and $\beta$  is stabilized by $\tau_2\tau'_2=(03)(12)(01)(24)=(02413)=\sigma_5^2$ and hence also by $\sigma_5$.

Finally, $\alpha'$ is stabilized  by $\tau'_2\tau_4^2=(01)(24)(03)(24)=(01)(03)=(031)$.  But we also have 
$\sigma_5^{-1}\sigma_3\sigma_5=(04321)(142)(01234)=(031)$.
\end{proof}

From now on we assume that $K$ and $(\tau_6, \tau_4, \tau_2)$ are as in Lemma \ref{lemma:goodchoice}.

\begin{lemma}\label{lemma:intersection1}
Every $\Af_5$-translate of $\alpha'$ meets $\alpha$ transversally  in at most one point.
Similarly, exactly three $\Af_5$-translates of $\beta$ meet  $\alpha$ resp.\  $\alpha'$, and they do so simply in at most one point.
\end{lemma}
\begin{proof}
The $\Af_5$-translates of $\alpha'$ meet the fundamental segment $[p_4, p_2]$ of $\alpha$  in $p_4$ only and through 
that point passes just one member, namely $\alpha'=\tau_4\alpha$. It then follows that the $\Sf_5$-translates of $\alpha$ distinct from $\alpha$ meeting $\alpha$ is the collection $\{\sigma_3^i\alpha'\}_i$.  These are pairwise distinct, proving the first assertion.  

The property regarding $\beta$ is proved in a similar fashion. The cyclic group generated by $\sigma_5$ resp.\  $\sigma_3$ is the $\Af_5$-stabilizer of $\vec\alpha$ resp. $\vec\beta$. Any $\Af_5$-translate  of $\beta$ which intersects $\alpha$ is of the form $\sigma_3^i\beta$ for some $i\in\Zb/3$. The $\Af_5$-stabilizer of $\sigma_3^i\vec\beta$ is generated by $\sigma_3^i\sigma_5\sigma_3^{-i}$. We have 
$\sigma_3\sigma_5\sigma_3^{-1}=(142)(01234)(124)=(04132)$ and $\sigma_3^{-1}\sigma_5\sigma_3 =(124)(01234)(142)=(02431)$ and neither is a power of $\sigma_5=(01234)$. So these stabilizers  are pairwise distinct. Hence so are the  $\{\sigma_3^i\vec\beta\}_{i\in\Zb/3}$, so that $\beta$ meets $\alpha$ in at most one point. 

Changing the orientation of $C_o$ has the effect of replacing $\tau_4$ by its inverse, and this exchanges $K$ and $K^*$, $\alpha$ and $\alpha'$, but preserves $\beta$. So  $\beta$ meets $\alpha'$ in at most one point.  
\end{proof}

Note that  $\vec\alpha\cdot\vec\alpha'=1$,  $\vec\beta\cdot \vec\alpha=1$ and $\vec\beta \cdot \vec\alpha'=1$. 
They define classes  in $H_1(C_o)$ which we continue to denote by the same symbol.
We write $\Delta (\alpha)$ for the $\Af_5$-orbit of $\vec\alpha$ in $H_1(C_o)$. This is a set of $10$ antipodal pairs. We define 
$\Delta (\alpha')$ and $\Delta (\beta)$ likewise: these are  sets of $10$ resp.\ $6$  antipodal pairs.

The following proposition gives us  the structure of $H_1(C_o)$ as a symplectic $\Sf_5$-module that we need
in order to determine the monodromy group. The group $\Sf_5$ acts in both $E_o$ and $H_1(C_o)$, but elements of   $V_o(C_o)=\Hom_{\Z A_5}(E_o,H_1(C_o))$  will rarely be $\Sf_5$-equivariant. This gives therefore   rise to an anti-involution 
$\iota$ in $V_o(C_o)$: for $\tau\in \Sf_5$,  the element $\iota(v):= \tau v\tau^{-1}$ is also in $V_o(C_o)$ and  only depends on the image of $\tau$ in $\Sf_5/\Af_5\cong \Zb/ 2$. The resulting involution is anti-linear with respect to the Galois involution in $\Oc_o$: for $\lambda\in \Oc_o$ and $v\in V_o(C_o)$, we have  $\iota (\lambda v)=\lambda'\iota (v)$.

\begin{proposition}\label{prop:symplecticbasis}
There exists a basis $(v, v')$ of $V_o(C_o)$  with the following properties:
\begin{enumerate}
\item[(i)]  The map $E_o^2\to H_1(C_o)$ given by $(a_1, a_2)\mapsto v(a_1)+X^3v'(a_2)$ is an isomorphism of $\Zb\Sf_5$-modules which maps each summand onto a Lagrangian submodule of 
$H_1(C_o)$, and is symplectic in the sense that $A(v, X^3v')=1$.
\item[(ii)] The anti-involution $\iota$ in $V_o(C_o)$ takes $(v, v')$ to $\pm(v', v)$ (we leave the sign as an unknown here). 
\item[(iii)] We have  $v(\Delta_c)=\Delta(\alpha)$,   $v'(\Delta_c)=\Delta(\alpha')$  and  $(v-v')(X^3\Delta_\ir)=\Delta(\beta)$.  
\end{enumerate}
\end{proposition} 

The proof of Proposition~\ref{prop:symplecticbasis} will use the following Lemma.

\begin{lemma}\label{lemma:n=-1}
Suppose that  $\Delta\subset E_\Q$ consists of $10$ (resp.\ 15) antipodal pairs and has the property that  $s$ takes on  $\Delta\times \Delta_c$ 
values in $\{-1,0,1\}$. In the second case, assume also that at most three antipodal pairs of $ \Delta$ are not perpendicular to a member of $\Delta_c$.
Then $\Delta$ equals $X^{-3}\Delta_c$ (resp. $\Delta_\ir$ or $X^{-3}\Delta_\ir$).
\end{lemma}
\begin{proof}
Since $\Delta_c$ generates $E_o$ and $s$ is unimodular on $E_o$, our assumption implies that $\Delta\subset E_o$.
The fact that $\Delta_c$ resp.\  $\Delta_\ir$ generates $E_o$ also implies that $\Delta$ is the image of $\Delta_c$ resp.\  $\Delta_\ir$ under a $\Af_5$-equivariant homomorphism $E_o\to E_o$,  so is given by a scalar $\lambda\in \Oc_o$. This scalar is unique up to sign, for any element of $\Oc_o^\times$ which leaves $\Delta$ invariant will be of finite order and hence equal to $\pm 1$.

A straightforward computation shows that $X^{-3}=2X-3$ sends the vector $e+e_0+e_1$ to $e_2-e_3+e_4$.  Noting that  $\{1, 2X-3\}$ is a $\Zb$-basis of
$\Oc_o$, we write $\lambda$ out on this basis: $\lambda=p+ qX^{-3}=p+q(2X-3)$ with $p, q\in\Z$.  

Assume now that $\Delta=\lambda\Delta_c$. We have  $e+e_1+e_2\in\Delta_c$ and so by our assumption
\begin{align*}
s(\lambda (e+e_0+e_1), e+e_0+e_1) &=s(p(e+e_0+e_1)+q(e_2-e_3+e_4), e+e_0+e_1) = 3p,\\
s(\lambda (e+e_0+e_1), e+e_1+e_2) &=s(p(e+e_0+e_1)+q(e_2-e_3+e_4), e+e_1+e_2) =2p+q
\end{align*}
both lie in $\{-1,0,1\}$. It follows that $p=0$ and $q=\pm 1$, so that  $\lambda=\pm X^{-3}$.

If $\Delta=\lambda\Delta_\ir$, then $\lambda (e+e_0+e_1)=p(e+e_0+e_1)+q(e_2-e_3+e_4)$
has  $s$-inner product of absolute value $\le 1$ with the elements of $\Delta_c$. This means that $|p|\le 1$ and $|q|\le 1$. 
It remains to show that $(p,q)$ is either  $(\pm 1,0)$ or $(0,\pm 1)$. If 
$(p,q)\not=(0,0)$, then for every $i\in\Zb/5$, 
\[
s(\lambda e_i, e+e_0+e_1)= s(e_i, \lambda (e+e_0+e_1))= s(e_i, p(e_2-e_3+e_4)+ q(e+e_0+e_1)\not=0, 
\]
so that at least  $5$ antipodal pairs in $\Delta$ would be not orthogonal to $e+e_0+e_1$. This we excluded. \end{proof}

\begin{proof}[Proof of Proposition \ref{prop:symplecticbasis}]
Recall that  $e+e_0+e_1$ is stabilized by  $\sigma_3$ and that its $\Af_5$-orbit generates $E_o$. Since $H_1(C_o; \Q)$ is isotypical
as a $\Q\Af_5$-module (of type $E_\Q$), it then follows that there exists a $v\in V_o(C_o)_\Qb= \Hom_{\Z \Af_5}(E_o, H_1(C_o; \Qb))$ such that
$v(e+e_0+e_1)=\vec\alpha$. Then $v$ will take its values in $H_1(C_o)$ and so will lie in  $V_o(C_o)$.
By Lemma \ref{lemma:goodchoice}, the vector  $\vec\alpha'\in H_1(C_o)$ is stabilized by $\sigma_5^{-1}\sigma_3\sigma_5$. Since $\sigma_5^{-1}\sigma_3\sigma_5$
stabilizes $\sigma_5^{-1}(e+e_0+e_1)$, it follows that  there exists a $v'\in V_o(C_o)$ such that $v'\sigma_5^{-1}(e+e_0+e_1)=\vec\alpha'$. 

Consider the pairing $(x, y)\in E_o\times E_o\mapsto  v(x)\cdot v'(y)$. This pairing is $\Af_5$-invariant.
We saw in Lemma \ref{lemma:intersection1} that it  takes on $\Delta_c\times\Delta_c$ values only in $\{-1,0, 1\}$. It then follows from 
Lemma \ref{lemma:n=-1}  that $v(x)\cdot v'(y)= \pm s(x,X^{-3}y)$ when $x,y \in \Delta_c$. In order to determine the sign, we note that  
$v(e+e_0+e_1)\cdot v'\sigma_5^{-1}(e+e_0+e_1)=\vec\alpha\cdot\vec\alpha'=1$ and  
\begin{multline*}
s(e+e_0+e_1, X^{-3}\sigma_5^{-1}(e+e_0+e_1))=\\=s(X^{-3}(e+e_0+e_1),\sigma_5^{-1}(e+e_0+e_1))=
s(e_2-e_3+e_4, e+e_4+e_0)=1
\end{multline*}
This shows that $v(x)\cdot v'(y)= s(x,X^{-3}y)$. Equivalently: 
$A(v, X^3v')=1$.  So $(v, X^{3}v')$ is a symplectic basis of $V_o(C_o)$ in the sense that  the $\Af_5$-equivariant map $E_o\oplus E_o\to H(C_o; \Zb)$ defined by $(v, X^3v')$ pulls back the intersection pairing on  $H(C_o; \Zb)$ to the symplectic pairing on $E_o\oplus E_o$ defined by $s$. Since the latter is unimodular (since $s$ is), this also implies that  the map $E_o\oplus E_o\to H(C_o; \Zb)$ is an isomorphism of symplectic modules.
  
 According to Lemma
\ref{lemma:generators},  $\Delta_c$ is $\Sf_5$-invariant. On the other hand, $\tau_4\in \Sf_5\ssm\Af_5$ takes 
$\alpha$ to $\alpha'$, and so it follows that
that $\iota (v)=\tau_4 v\tau_4^{-1}$ is an element of $V_o(C_o)$ that takes $\Delta_c$ to $\Delta (\alpha')$. But $v'$ also has this property. It is the only element of  
$V_o(C_o)$ with this property up to sign, for any two elements of $V_o(C_o)$ for which the images of $\Delta_c$ coincide will differ by a factor in $\Oc_o^\times$ which is of finite order, in other words, will differ by a sign. As we are only interested in the effect of conjugation with $\iota$,  this sign is unimportant for us  and we leave it as an unknown: we have  $\iota(v)=\pm v'$. 
The map $\iota$ is an antilinear involution of $V_o(C_o)$ and so $v=\iota^2(v)=\iota (\pm v')$, which shows that
$\iota(v')= \pm v$. 

Finally,  since the $\Af_5$-stabilizers of  $e$  and $\vec\beta$ are  generated by $\sigma_5$, there exists a 
$u\in V_o(C_o)$ such that $u(e)=\vec\beta$.
The set $\Delta (\beta)$ consists of $6$ antipodal pairs in $H_1(C_o)$. Lemma \ref{lemma:intersection1}  tells us that the hypotheses
of Lemma \ref{lemma:n=-1} are fulfilled (second case): $(x, y)\in \Delta_\ir\times \Delta_c\mapsto u(x)\cdot v(y)$  takes its values 
in $\{-1,0,1\}$ with for a given $y\in \Delta_c$, a nonzero value occurring for at most three antipodal pairs in $\Delta_\ir$.  
 It then follows that either $u(x)\cdot v(y)=\pm s(x,y)$ or  $u(x)\cdot v(y)=\pm s(x,X^{-3}y)$.
We have $u(e)\cdot v(e+e_0+e_1)=\vec\beta\cdot\vec\alpha=1$, $s(e,e+e_0+e_1)=1$ and $s(e, X^{-3}(e+e_0+e_1))=s(e,e_2-e_3+e_4)=0$.
It follows that $u(x)\cdot v(y)=s(x,y)$, in other words, $A(u,v)=1$. 

The same argument works  if we replace $v$ by  $v'$:  $u(x)\cdot v'(y)=\pm s(x,y)$ or  $u(x)\cdot v'(y)=\pm s(x,X^{-3}y)$.
Since we  have $u(e)\cdot v'(e+e_4+e_0)=\vec\beta\cdot\vec\alpha'=1$, $s(e,\sigma_5^{-1}(e+e_0+e_1))=s(\sigma_5(e),e+e_0+e_1)=1$ and 
\begin{multline*}
s(e, X^{-3}\sigma_5^{-1}(e+e_0+e_1))=s(e, \sigma_5^{-1}X^{-3}(e+e_0+e_1))=\\=
s(\sigma_5(e),e_2-e_3+e_4)= s(e,e_2-e_3+e_4)=0,
\end{multline*}
it follows that $u(x)\cdot v'(y)=s(x,y)$, so that $A(u,v')=1$. Since $A(v,v')=X^{-3}$, this proves that $u=X^3(v-v')$ and so $\Delta(\beta)=(v-v')(X^3\Delta_\ir)$. 
\end{proof}



\begin{remark}\label{rem:symplecticbasis}
The second property implies that if an endomorphism $T$ of $V_o(C_o)$ has on the basis $(v, v')$ the matrix 
$\left( \begin{smallmatrix}a &b \\ c & d\end{smallmatrix}\right)$, then
$\iota T\iota$ has the matrix $\left( \begin{smallmatrix}d' &c' \\ b' & a'\end{smallmatrix}\right)$.
\end{remark}

\begin{remark}\label{rem:loops}
We can realize $\Delta(\alpha)$, $\Delta(\beta)$ and their $\iota$-transforms as sets of vanishing cycles as follows.
Consider in  $P$ the union of the four arcs  $\alpha, \alpha', \beta, \beta'$. So this is the union of
the unit circle and its center line  $[-1,1]$. It contains all the order $2$-orbifold points. By moving these points 
we deform $C_o$ as  a $\Af_5$-curve. In the present case, each of the four arcs determines a simple
way to do this and  gives a path in $\Bc$  from $c_o$ to one of the points  $c_c, c'_c, c_\ir, c'_\ir$.
For the arc $\alpha$, we move the central point $0$ along the ray $[0,1]$ to its end point $1$ and for the arc
$\beta$ we move $1$ along the semi-circle  to its end point $-1$ (we fix the other orbifold points)
and we do the obvious analogue for $\alpha'$ and $\beta'$. 
The arcs $\gamma_\alpha$, $\gamma_\beta$,  $\gamma_{\alpha'}$,   $\gamma_{\beta'}$ in $\Bc$ thus obtained have 
the property that they do not meet
away from $c_o$ and avoid discriminant points except at the end point. We assume the labeling such that the end 
points are  $c_c$, $c'_c$, $c_\ir$ and $c'_\ir$ 
respectively. Notice that $\iota$ exchanges  the items of $\gamma_\alpha, \gamma_{\alpha'}$ and  
$\gamma_\beta, \gamma_{\beta'}$ 

To the arc $\gamma_\alpha$ from $c_o$ to $c_c$ there is associated element   $[\gamma_\alpha]\in\pi_i(\Bc^\circ, c_o)$ represented  by 
a positive simple loop based at $c_o$ around the end point $c_c$ of $\gamma_\alpha$ and similarly for the other arcs.
The four elements $[\gamma_\alpha], [\gamma_\beta], [\gamma_{\alpha'}], [\gamma_{\beta'}]$ generate $\pi_i(\Bc^\circ, c_o)$ freely and 
$[\gamma_\alpha][\gamma_\beta][\gamma_{\alpha'}][\gamma_{\beta'}]$ (read from right to left) represents a negative simple loop around $c_\infty$.
\end{remark}

\subsection{The local system of isogeny modules of the Wiman-Edge pencil}\label{subsect:locsystem} 
The $V_o(C_t)$ define a local system
$\Vb_o$  of symplectic $\Oc_o$-modules of rank $2$ over  $\Bc^\circ$. The involution $\iota$ of $\Bc^\circ$ (that is given by 
precomposition with a non-inner automorphism of  $\Af_5$) lifts to an isomorphism between the pull-back $\iota^*\Vb_o$ and the twist  of
$\Vb_o$  as a $\Oc_o$-module by the Galois involution. In other words, the involution $\iota$ lifts in an anti-linear manner to $\Vb_o$. 
As proved in Theorem 1.1 (see also Remark 2.5) of \cite{FL}, there is an identification $\Bc^\circ=\Gamma\bs \Hf$ with $\Gamma\subset \PSL_2(\Z)$ being torsion 
free and so $\Vb_o$ pulls back to $\Hf$ as a trivial symplectic local system  with $\G$-action.   The  basis $(v, v')$ of $V_o(C_o)$ constructed in 
Proposition \ref{prop:symplecticbasis} extends to one of the pull-back of  $\Vb_o$ to $\Hf$ (so we use $c_o$ as our base point).
Now the  $\G$-action (and hence the monodromy of $\Vb_o$) is given by a  homomorphism  
\[
\rho: \Gamma\to \SL_2(\Oc_o) 
\]
that is compatible with the involutions named $\iota$ (it acts in $\SL_2(\Oc_o)$ as prescribed by Remark \ref{rem:symplecticbasis}. (So this 
yields a group homomorphism between the semi-direct products defined by these involutions; this can be understood as a monodromy representation of
a local system on the Deligne-Mumford stack $\Bc^\circ/\iota$.)
 
Our goal is to describe this monodromy  representation. We first do this  locally.

\subsubsection*{The cusps of $\SL_2(\Oc_o)$}
We observed that $\Oc/2\Oc\cong\Fb_4$ and that the Galois involution of  $\Oc$ induces in $\Oc/2\Oc$  the  Frobenius map  (so its fixed point set is the prime field $\Oc_o/2\Oc=\Fb_2$).  Reduction modulo $2$ defines a homomorphism $\SL_2(\Oc)\to \SL_2(\Fb_4)$.
It is surjective and the permutation representation of $\SL_2(\Fb_4)$ on $\Pb^1(\Fb_4)$ identifies  $\SL_2(\Fb_4)$ with 
the alternating group $\Af_5$. (It is also known that the  full permutation group of $\Pb^1(\Fb_4)$ is the semi-direct product of 
$\SL_2(\Fb_4)$ and the Frobenius.) 

We shall write  $\SL_2(\Oc)[2]$ for the kernel of  $\SL_2(\Oc)\to \SL_2(\Fb_4)$. 
Since $\Pb^1(\Fb_4)$ has $5$ elements, $\SL_2(\Oc)[2]$  has as many cusps ($=\SL_2(\Oc)[2]$-orbits in $\Pb^1(K)$).   These are represented by 
$[1:0]$, $[0:1]$, $[1:1]$, $[X:1]$, $[1:X]$. Note that the involution $I : (x_0,x_1)\mapsto (x_1, x_0)$  exchanges $[1:0]$ and $[0:1]$ and $[X:1]$ and 
$[1:X]$, whereas the Frobenius only  exchanges $[X:1]$ and $[1:X]$ (for $[X^2:1]=[1: X^{-2}]=[1:-X+2]$ and $[1:X]$ define the same element of $\Pb^1(\Fb_4)$). 

It is  clear that
$\SL_2(\Oc_o)$ is the preimage of $\SL_2(\Fb_2)$, when regarded as a subgroup of $\SL_2(\Fb_4)$. The subgroup  $\SL_2(\Fb_2)\subset \SL_2(\Fb_4)$  has two orbits in 
$\Pb^1(\Fb_4)$, namely $\{[1:0],[1:1], [0:1]\}$  and $\{[X:1], [1:X]\}$,  and so  $\SL_2(\Oc_o)$  has only 2 cusps  (which we shall denote $\infty_0$ resp.\  $\infty_X$), 
both of which are invariant under the involution $I$. This has the following implication.

\begin{lemma}\label{lemma:dichotomy}
A rank one $\Oc_o$-submodule $L\subset \Oc_o^2$ which is primitive in the sense that $\Oc_o^2/L$ is torsion free,  is a $\SL_2(\Oc_o)$-transform of  either 
\begin{description}
\item[(type $\infty_0$)] the first summand of $\Oc_o^2$ (the associated $\SL_2(\Oc_o)$-cusp is $\infty_0$), or
\item[(type $\infty_X$)] the image of $a\in\Oc\mapsto (2a, 2Xa)\in \Oc_o^2$ (the associated $\SL_2(\Oc_o)$-cusp is $\infty_X$).
\end{description}
\end{lemma}
\begin{proof}  By regarding  $\Q\otimes_\Z L=K\otimes_{\Oc_o} L$ as a $K$-linear subspace of $K^2$ of dimension one, we get an element of $\Pb^1(K)$. Since
the  $\SL_2(\Oc_o)$-orbits in $\Pb^1(K)$ are represented by $[1:0]$ and $[1:X]$, we can assume that either $\Q\otimes_\Z L$ is the first summand of  $K^2$ or the graph of $X$. In the first case, it is clear that  $L$ is the first summand of $\Oc_o^2$. 
In the second case, we note that if $u\in \Oc_o$ is such that $Xu\in \Oc_o$, then by writing $u$ as an integral linear combination of $1$ and $X$, we find that $u\in 2\Oc$. Conversely, every element of $2\Oc$ has that property. 
\end{proof}

\subsubsection*{The cusps of the principal level 2 subgroup of $\SL_2(\Oc_o)$}
We next consider the mod 2 reduction of $\Oc_o$ and  $\SL_2(\Oc_o)$.  A $\Zb$-basis of $\Oc_o$ consists of  $1$ and $Y:=2X$.  Since
$Y^2=4X+4=2Y+4\in 2\Oc_o$, it follows that $\Oc_o/2\Oc_o\cong \Fb_2[Y]/(Y^2)$ as a ring. Its group of units is $\{1, 1+Y\}$. 
So a nonzero submodule of $(\Oc_o/2\Oc_o)^2$ generated by a single element  are the ones generated by 
$(1,0), (1, Y),  (1,1),  (1+Y,1), (Y,Y), (Y,1),  (0,1)$. These seven submodules are pairwise distinct. 
Note that the involution $(x_0,x_1)\mapsto (x_1, x_0)$  exchanges the items of $\{[1:0],[0:1]\}$ and $\{[Y:1], [1:Y]\}$ and fixes the other three $[1:1]$,   $[1+Y:1]$, $[Y:Y]$.  The reduction homomorphism $\SL_2(\Oc_o)\to \SL_2(\Oc_o/2\Oc_o)$ is onto.

\section{Arithmeticity of the monodromy}
\label{sect:monodromy2}

In \S~\ref{sect:monodromy1} we proved that the monodromy representation of the Wiman-Edge pencil 
has target $\SL_2(\Oc_o)$, giving a representation $\rho:\Gamma\to\SL_2(\Oc_o)$.  The goal of this section is to prove Theorem~\ref{theorem:arithmeticity}, that the image of $\rho$ has finite index in $\SL_2(\Oc_o)$.  The first step in doing this is to compute the image under $\rho$ of the generators of $\Gamma$.  As explained above, $\Gamma$ is generated by loops around the cusps, i.e.\ the degenerations of the pencil.  We start by applying classical Picard-Lefschetz theory to compute the conjugacy classes of these local monodromies.   Computing them on the nose requires more work, which we do later in the section.


\subsection{The monodromy around a cusp} In order to gain a better understanding of what $\rho$ is like,  let us recall that $\Gamma$ is a free group and is generated by simple loops around the   punctures of the $5$-punctured sphere $\Bc^\circ$. We therefore concentrate on  the monodromy around each puncture. 

Assume that $C_s$ is singular, and choose a disk-like neighborhood $U$ of $s$ in $\{s\}\cup\Bc^\circ$ (so that $C_s\subset  \Cs_U$ is a homotopy equivalence). Choose also $\eta\in U\ssm\{s\}$
and write $C$ for $C_\eta$. Then
the natural map $H_1(C)\to H_1(\Cs_U)\cong H_1(C_s)$ is onto, and the kernel is a $\Af_5$-invariant isotropic  sublattice. If $G_s$ denotes the  dual intersection graph of $C_s$ then there is a natural homotopy class of maps $C_s\to G_s$. Since the irreducible components
of the normalization of $C_s$ are all of genus zero, this homotopy class induces an isomorphism on the first (co)homology. We note that in each case 
$H_1(G_s)$ is free of rank $6$, and so the kernel of  $H_1(C)\to H_1(C_s)\cong H_1(G_s)$ is in fact a primitive Lagrangian sublattice. The intersection pairing identifies this kernel with the dual $H^1(G_s)$ of $H_1(C_s)$,   so that we have an exact sequence of $\Z \Af_5$-modules
\begin{equation*} \tag{\text{vanishing sequence}}
0\to H^1(G_s)\to H_1(C)\to  H_1(G_s)\to 0\quad ,
\end{equation*}
where $H^1(G_s)\to H_1(C)$ is the composition of 
$H^1(G_s)\to H^1(C)$ with the isomorphism $H^1(C)\cong H_1(C)$ defined by the intersection pairing.  The sequence is preserved by the monodromy operator of the family $\Cs_{U\ssm \{s\}}$, with the monodromy acting nontrivially only on the middle term. 
Its difference with the identity (which is also called the \emph{variation of the monodromy}) is therefore given by a homomorphism
\[\tag{\text{variation of the local monodromy}}
\nu_s: H_1(G_s)\to H^1(G_s) 
\] 
of $\Z\Af_5$-modules. This homomorphism can be read off from $G_s$.
To see this,  we recall that the  monodromy is given by the classical Picard Lefschetz formula. Each node of $C_s$ determines a vanishing circle on $C$ up to isotopy, and hence, after orienting it, an element of  $H_1(C)$ up to sign (a \emph{vanishing cycle}).  We denote the  collection of vanishing cycles by $\Delta_C\subset H_1(C)$. It is clear from the preceding that  $\Delta_C$ lies in  the image of $H^1(G_s)\to H^1(C)\cong H_1(C)$.  In fact, $\Delta_C$ generates that image. The monodromy around $C_s$  is a multi Dehn twist which  acts on $H_1(C)$ as:
\[
\textstyle T_s: x\in H_1(C)\mapsto x+\sum_{\delta\in \Delta_C/\{ \pm 1\}}(\delta\cdot x)\delta\in H_1(C)
\]
(the sum makes sense because replacing $\delta$ by $-\delta$ does not alter $(\delta\cdot x)\delta$).  We now see that  if we identify 
$H_1(G_s)$ with the dual of $H^1(G_s)$, then $\nu_s\in \Hom ( H_1(G_s),  H^1(G_s))\cong H^1(G_s)\otimes H^1(G_s)$ is represented by 
the symmetric tensor
\[
\textstyle t_s:=\sum_{\delta\in \Delta_C/\{ \pm 1\}} \delta\otimes\delta\in H^1(G_s)\otimes H^1(G_s).
\]
We shall refer to $t_s$ as the \emph{variation tensor}. It is the sum over the squares of the edges of $G_s$ and so canonically associated with $G_s$. It is of course also $\Af_5$-invariant. By construction $T_s(x)-x$ is obtained by contracting $t_s$ on the right with the image $[x]$ of $x$ in $H_1(G_s)$ 
and regard the resulting element of  $H^1(G_s)$ as sitting in $H_1(C)$ via Poincar\'e duality. 
In order to determine $t_s$  in each case, we first note  that the inverse form of $s$ (i.e., the quadratic form on $E_o^\vee$) is the tensor  
\[
\textstyle \check{s}=e\otimes  e+\sum_{i\in \Z/(5)} e_i\otimes e_i\in E_o\otimes E_o.
\]
 
\begin{proposition}[{\bf Variation tensors of the singular fibers}]
\label{prop:cuspmonodromy}
For a singular fiber $C_s$ of the Wiman-Edge pencil, its homology as a $\Af_5$-module and its variation tensor are as follows:

\begin{enumerate}
\item When $C_s$ is irreducible, there exists an isomorphism $v_s: E_o\cong H^1(G_s)$ of $\Z\Af_5$-modules (so the associated cusp is $\infty_0$) such that  $t_s=v_s\otimes v_s(\check{s})$. 

\item When $C_s$ consists of five conics, there exists an isomorphism $v_s: E_o\cong H^1(G_s)$ (so the associated cusp is $\infty_0$) such that  $t_s=(3+4X)v_s\otimes v_s(\check{s})$.

\item When $C_s$ consists of ten lines, there exists an isomorphism $v_s: E\cong H^1(G_s)$  (so the associated cusp is $\infty_X$) such that  $t_s=(4+2X)v\otimes v(\check{s})$.
\end{enumerate}
 \end{proposition}
 
\begin{proof}[Proof when $C_s$ is irreducible]
In this case $G:=G_s$ has one vertex with $6$ loops attached.  Choose an orientation of each loop 
of $G$. This  selects a vanishing cycle from each of our six antipodal pairs and the associated   classes in $H^1(G)$ make up a basis and 
the variation of the monodromy assigns to the oriented loop the associated vanishing cycle.
An $\Af_5$-isomorphism $v: E_0\cong  H^1(G)\subset H_1(C)$ is defined by assigning to $e$ one such vanishing cycle. 
It is clear that then $t_s$ is as asserted.
\end{proof}

\begin{proof}[Proof when $C_s$ is the union of  5 conics]
In this case $G_s$ is the $K_5$ graph (which has  $\Sf_5$ as its automorphism group).
It has  $20$  oriented edges  and   $\Af_5$ acts transitively on this set. We have in fact an  $\Af_5$-equivariant  bijection between the order $3$-elements in $\Af_5$ (the conjugacy class of $3$-cycles) and the oriented edges of $K_5$ by assigning to  the $3$-cycle $h=(\tau_1\tau_2 \tau_3)$ the oriented edge $[\tau_4,\tau_5]$ which is characterized by the property that  the permutation $i\mapsto \tau_i$ is even. Note that this assigns to the inverse element $(\tau_1\tau_3\tau_2)$ the oppositely oriented edge $[\tau_5,\tau_4]$. Thus the group $Z^1(K_5)$ of simplicial $1$-cochains on $K_5$ is identified with  the $\Z$-module of rank $10$ generated by the order $3$ elements $h\in\Af_5$, subject to the relations $h+h^{-1}=0$.  Its $\Z\Af_5$-module structure is defined by conjugation. The $1$-coboundary submodule  $B^1(K_5)$ has rank $4$ and  is spanned by the vertices subject to the relation that the sum of the vertices is zero. So $B^1(K_5)_\C$ is the reflection representation; it is irreducible as a $\Af_5$-module.

According to Lemma \ref{lemma:generators}, 
the $\Af_5$-orbit $\Delta_c$ of $e+e_{0}+e_1\in E_o$ consists of the 10 opposite pairs   $\{\pm(e+e_{i}+e_{i+1})\}_{i\in \Zb/(5)}, \pm(e_i -e_{i-2}-e_{i+2})\}_{i\in \Zb/(5)}$, spans $E_o$ over $\Zb$ and $\sigma_3\in\Af_5$ generates the $\Af_5$-stabilizer of $e+e_0+e_1$. 
A $\Z\Af_5$-module epimorphism $Z^1(K_5)\to E_o$ is then defined by demanding that it takes the oriented edge fixed by $\sigma_3$ to  $e+e_0+e_1$.  Since $E_\Q$ is irreducible as a $\Q\Af_5$-module, the image of $B^1(K_5)_\Q$ in $E_\Q$ is zero. So $B^1(K_5)$ is contained in the kernel of $Z^1(K_5)\to E_o$. But $B^1(K_5)$ is a primitive submodule of $Z^1(K_5)$ (for $H^1(K_5)$ is torsion free) and  has the same rank as this kernel. So it is equal to the kernel and we have an induced isomorphism $\Z\Af_5$-module isomorphism $v_s: E_o\to H^1(K_5)$. 

The associated quadratic tensor is the image under $v_s\otimes v_s$  of 
\[
\textstyle \sum_{i\in \Zb/(5)} (e +e_{i-1}+e_{i+1})^{\otimes 2} +\sum_{i\in \Zb/(5)} (e_i -e_{i-2}-e_{i+2})^{\otimes 2}.
\]
If we write this tensor as $u\otimes e+\sum_i u_i\otimes e_i$ we find that 
$u=5e +\sum_i e_{i-1}+ \sum_i e_{i+1}=3e+2(e+\sum_i e_i)=3e+4\varepsilon=(3+4X)(e)$. Likewise we find that $u_i=(3+4X)(e_i)$.
\end{proof}

\begin{proof}[Proof when $C_s$ is the union of  10 lines]
In this case $G_s$  is the Petersen graph $P$. Recall that the vertices of $P$ are indexed by the $2$-element subsets of 
$\Z /5$ (a set of size $10$) and that two such $2$-element subsets span an edge if and only they are disjoint (a set of size $15$). 
This makes it plain that $\Af_5$ acts transitively on its set of oriented edges (a set of $15$ antipodal pairs), so that the stabilizer of an oriented 
edge is of order $2$. The elements of order $2$ in $\Af_5$ make up a single conjugacy class, and a given order $2$ element preserves just  one 
edge in an orientation preserving manner. The centralizer of  that  element is a copy of $\Z/2\oplus\Z/2$; it preserves the edge, but may reverse 
orientation.  

Similarly,  $e+e_o\in E$ has as its $\Af_5$-stabilizer a subgroup of order $2$ (namely $\sigma_2$). It is clear that if $g\in \Af_5$
maps $e+e_o$ to $-(e+e_o)$, then it must centralize $\sigma_2$. It follows that there exists  a $\Af_5$-equivariant bijection of the set of oriented edges of $P$ onto the $\Af_5$-orbit of $e+e_o$ with the property that  orientation reversal corresponds to taking antipode.  In view of Lemma  
\ref{lemma:generators} this homomorphism is onto. 

Recall that $E_\C$ is an $\C \Af_5$-module isomorphic to $I_\C\oplus I'_\C$, where $I_\C$ and $I'_\C$ are
irreducible of degree  $3$. So to prove that $Z^1(P)\to E$ factors through an isomorphism $H^1(P)\to E$, it suffices to show that the coboundary space $B^1(P; \C)$, does not contain a copy of $I_\C$ or $I'_\C$. As we observed in the proof of Lemma 2.1 of part I \cite{DFL},  the vertex set of $P$ spans a $\C\Af_5$-module  which decomposes into a trivial representation and two irreducible representations of dimension $4$ and $5$. Hence the  latter two will span $B^1(P; \C)$. In particular, neither  $I_\C$ nor $I'_\C$ appears in $B^1(P; \C)$ and so we have an   isomorphism $H^1(P)\to E$. We denote by $v_s: E\cong H^1(P)$ its inverse.

The associated quadratic tensor is then the  image under $v_s\otimes v_s$  of 
\[
\textstyle \sum_{i\in \Zb/(5)} (e+e_i)^{\otimes 2} +\sum_{i\in \Zb/(5)} (e_i +e_{i+1})^{\otimes 2}  +\sum_{i\in \Zb/(5)} (e_{i-1}-e_{i+1})^{\otimes 2}.
\]
Proceeding as in the previous case, we write this as $u\otimes e +\sum_i u_i\otimes e_i$ and find that 
$u=5e+\sum_i e_i= 4e+2\varepsilon = (4+ 2X)(e)$ and likewise that $u_i=(4+ 2X)(e_i)$.
\end{proof}

\begin{remark}
The descriptions in Proposition~\ref{prop:cuspmonodromy} also tell us what the monodromy variations, or rather their   $\SL_2(\Oc_o)$-conjugacy classes, are in terms of $\Vb_o(\eta)$: as this endomorphism of $\Vb_o(\eta)$ is $\Oc_o$-linear, they are in the three cases  given by respectively the images of  $v\otimes v$,  $(3+4X)v\otimes v$, and $(4+2X)v\otimes v$ in $\Vb_o(\eta)\otimes_{\Oc_o}\Vb_o(\eta)$ (in the last case this element lies {\it a priori} only  
in $\Vb(\eta)\otimes_{\Oc}\Vb(\eta)$, but one checks that it actually lies in the image of $\Vb_o(\eta)\otimes_{\Oc_o}\Vb_o(\eta)$).

As we might expect, each of these three conjugacy classes is Galois invariant. In the first case this is obvious. To see this in the two other cases, recall that the group of units of $\Oc_o$ is generated by $X^3=2X+1$. So we can  change the representative of the conjugacy class by  conjugation with the diagonal matrix in $\SL_2(\Oc_o)$ with diagonal entries  $X^3$ and $X^{-3}$. Its effect on  $\binom{1}{0}\otimes_K\binom{1}{0}$ is multiplication with $X^{6}$. 
Since $(3+4X)'= 7-4X=X^{-6}(3+4X)$, it follows that the conjugacy class defined by the five conics is indeed Galois invariant.
This also the case for the conjugacy class  defined by the ten lines: $A:=(\begin{smallmatrix}
1 & -1\\
2 & -1\\
\end{smallmatrix})\in \SL_2(\Z)$ takes $\binom{1}{X'}=\binom{1}{1-X}$ to  $\binom{X}{1+X}=X\binom{1}{X}$ and hence takes 
$(4+2X')\binom{1}{X'}\otimes_K \binom{1}{X'}$ to 
\[(6-2X)X^2\binom{1}{X}\otimes_K \binom{1}{X}=(4+2X)\binom{1}{X}\otimes_K \binom{1}{X}.\] 
\end{remark}

\bigskip
We can be more precise in that we can obtain actual monodromies rather than just their conjugacy classes.
Recall that the involution of  the Wiman-Edge pencil  determines an involution $\iota$ of $\Bc$ with fixed points $c_o$ and $c_\infty$ representing the
Wiman curve resp.\ the union of ten lines. This involution is covered by an involution of $\Vb_o$ which is anti-linear: if $\Vb'_o$  
denotes the same local system but for which the $\Oc_o$-module structure has been precomposed with nontrival Galois element 
$X\mapsto X'=1-X$, then we have an identification $\Vb_o\cong \iota_*\Vb_o'$ such that applying this twice gives the identity.
The  singular fibers $\not=C_\infty$ come in two (unordered) pairs $\{C_\ir,C'_\ir\}$, $\{C_c,C'_c\}$ and lie over points denoted $c_\ir, c'_\ir=\iota(c_\ir)$ resp.\   $c_c, c'_c =\iota(c_c)$.  We focus on the affine line $\Bc\ssm\{c_\infty\}$. 

We observed in Remark \ref{rem:loops} that
the elements of  $\pi_1(\Bc^\circ, c_o)\cong\G$ defined by $\gamma_\alpha, \gamma_{\alpha'}$ and  $\gamma_\beta, \gamma_{\beta'}$ freely generate $\pi_1(\Bc^\circ, c_o)\cong\G$;  that  $\iota$ exchanges each pair; and that a simple negative loop around $c_\infty$ is represented by the product $[\gamma_\alpha][\gamma_\beta][\gamma_{\alpha'}][\gamma_{\beta'}]$. We shall abuse notation a bit by writing
$\rho(\gamma_\alpha)$ for  $\rho([\gamma_\alpha])$ and likewise for the other generators.

\begin{corollary}\label{cor:monodromies}
The monodromies satisfy the following identities:

\[
\rho(\gamma_\alpha)=\begin{pmatrix}
1 & -1+2X\\
0 & 1\\
\end{pmatrix},\quad 
\rho(\gamma_{\alpha'})=\begin{pmatrix}
1 & 0\\
1-2X & 1\\
\end{pmatrix}
\]
and
\[
\rho(\gamma_{\beta})=
\begin{pmatrix}
1+X^3& X^3\\
-X^3 &  1-X^3
\end{pmatrix}, \quad 
\rho(\gamma_{\beta'})=\begin{pmatrix}
1+X^{-3}& X^{-3}\\
-X^{-3} & 1-X^{-3}
\end{pmatrix}.
\]
These elements determine the monodromy representation of $\Vb_o$ and generate its monodromy group. 
This monodromy group fixes the map $\Oc_o^2\to \Fb_5$ defined by taking the  symplectic product  with $\binom{1}{-1}(=v-v')$ followed by  
the reduction $\Oc_o\to\Fb_5$. Up to a scalar this is the map $V_o(C_o)\to \Fb_5$ found in Corollary \ref{cor:toN5}.
\end{corollary}
\begin{proof}
The set of vanishing cycles for the degeneration along $\gamma_\alpha$ is $v(\Delta_c)$. According to Proposition \ref{prop:cuspmonodromy}, the variation tensor of $\rho(\gamma_\alpha)$ is 
then  $(3+4X)\check{s}(v\otimes v)$. Now note that $(3+4X)A(v,v')=(3+4X)(-3+2X)=-1+2X$. Hence $\rho(\gamma_\alpha)(v)=v$ and 
\[
\rho(\gamma_\alpha)(v')=v' +(3+4X)A(v,v')v=v'+ (-1+2X)v 
\]
so that $\rho(\gamma_\alpha)$ is as asserted. Proposition \ref{prop:symplecticbasis} also shows that the set of vanishing cycles for the degeneration along $\gamma_\beta$
is the image of $\Delta_\ir$ under $u=X^{3}v-X^3v'$. We have $A(u, v')=A(u, v)=+1$ and so
\begin{gather*}
\rho(\gamma_\beta)(v)=v +u=(1+X^3)v-X^3v',\\
\rho(\gamma_\beta)(v')=v' + u=X^3v+(1-X^3)v',
\end{gather*}
which yields the matrix for $\rho(\gamma_{\beta})$. The matrices for   $\rho(\gamma_{\alpha'})$ and $\rho(\gamma_{\beta'})$ are then obtained using Remark \ref{rem:symplecticbasis} and the fact that $(-1+2X)'=1-2X$ and  $(X^3)'=-X^{-3}$.

We note that the images 
of $\rho(\gamma_\alpha)$ and  $\rho(\gamma_{\alpha'})$ are trivial in $\SL_2(\Fb_5)$. The images of  $\rho(\gamma_\beta)$ and  $\rho(\gamma_{\beta'})$
must generate the same one-parameter subgroup $U\subset \SL_2(\Fb_5)$, namely the additive copy of $\Fb_5$ defined by the quadratic tensor 
$\binom{1}{-1}\otimes \binom{1}{-1}$. So the monodromy fixes pointwise the line in $\Fb_5\otimes V_o(C_o)$ spanned by $v-v'$ and it is the only line spanned by that property. So it is  the one determined  by Corollary \ref{cor:toN5}.
\end{proof}

\begin{question}\label{quest:}
It is well-known that  the reduction  homomorphism  $\SL_2(\Z)\to \SL_2(\Fb_5)$ is onto and hence so is  $\SL_2(\Oc_o)\to \SL_2(\Fb_5)$.
Are there no other restrictions on the monodromy group other than the one given in Corollary \ref{cor:monodromies}, in the sense that it  contains the kernel of the reduction map $\SL_2(\Oc_o)\to \SL_2(\Fb_5)$?
\end{question}

We note that  $\gamma_\beta\gamma_\alpha\gamma_{\beta'}\gamma_{\alpha'}$ represents a negative loop around $c_\infty$ and so its image under $\rho$ is in the conjugacy class of the  unipotent element defined by  the variation $-2(2+X)v_s\otimes v_s(\check{s})$ for some vector $v_s$ of type $\infty_X$ (as defined in Lemma \ref{lemma:dichotomy}). Hence the following proposition gives  an additional  check on our computations.

\begin{proposition}\label{prop:check}
Let $v_\infty:=v+Xv'$. Then $\rho (\gamma_\alpha\gamma_\beta\gamma_\alpha'\gamma_\beta')$ is given by the variation tensor 
$-2(2+X)v_\infty\otimes v_\infty(\check{s})$.  
\end{proposition}
\begin{proof}
We put $B:=\rho (\gamma_\alpha\gamma_\beta)$, so that 
$\rho (\gamma_\alpha\gamma_\beta\gamma_\alpha'\gamma_\beta')=B\iota B\iota$. 
We compute
\[
B=\rho (\gamma_\alpha)\rho (\gamma_\beta)=
\begin{pmatrix}
1 & -1+2X\\
0 & 1
\end{pmatrix}
\begin{pmatrix}
1+X^3& X^3\\
-X^3 &  1-X^3
\end{pmatrix}
=
\begin{pmatrix}
-1-2X & -3\\
-1-2X & -2X 
\end{pmatrix}.
\]
Then in view of Remark \ref{rem:symplecticbasis}, 
\[
\iota B\iota=\rho (\gamma_{\beta'})\rho (\gamma_{\alpha'})= 
\begin{pmatrix}
-2X' & -1-2X'\\
-3 & -1-2X'
\end{pmatrix}=
\begin{pmatrix}
-2+2X&-3+2X\\
-3&  -3+2X
\end{pmatrix}
\]
and so
\[
 B\iota B\iota=
\begin{pmatrix}
-1-2X & -3\\
-1-2X & -2X 
\end{pmatrix}.
\begin{pmatrix}
-2+2X&-3+2X\\
-3&  -3+2X
\end{pmatrix}
=
\begin{pmatrix}
7-2X  &  8-6X\\
-2+4X  & -5+2X
\end{pmatrix}
\]
It follows that
\[
B\iota B\iota-\mathbf{1}= 
2\begin{pmatrix}
3-X  & 4-3X\\
-1+2X  & -3+X
\end{pmatrix}
= 
-2(2+X)
\begin{pmatrix}
-X^{-2}&    X^{-3}\\
-X^{-1}  &  X^{-2}
\end{pmatrix}.
\]
Now note that $A(v_\infty, v)v_\infty,=-X^{-2}(v+Xv')= -X^{-2}v- X^{-1}v'$ and 
$A(v_\infty', v')v_\infty=X^{-3}(v+Xv')= X^{-3}v +X^{-2}v'$ and so the last matrix is indeed the  matrix of the endomorphism
$x\in V_o(C_o)\mapsto -2(2+X)A(v_\infty, x) v_\infty\in V_o(C_o)$.
\end{proof}

\subsection{Arithmeticity}

Now that we know the image of the generators of $\Gamma$ under the monodromy representation, we can prove arithmeticity of the monodromy group. 
We use a criterion due to Benoist-Oh, namely Theorem 1.1 of \cite{BO}.  That theorem gives the following as a special case.

\begin{theorem}[{\bf Benoist-Oh}]\label{theorem:BenoistOh}
Let $K$ be a real quadratic number field, $\Oc_K$ its ring of integers, and $\Omega< K$ a lattice. Let $\Lambda<\SL_2(\Oc_K)$ be the subgroup generated by a matrix of the form 
$
\left(
\begin{array}{cc}
a&b\\
c&d
\end{array}
\right)
$ with $c\neq0$, together with the set of matrices
\[ 
\Big\{\left(
\begin{array}{cc}
1&\omega\\
0&1
\end{array} \right)
: \omega\in \Omega\Big\}.
\] 
If $\sigma, \sigma': K\to \Rb$ are the two real embeddings of $K$, then the associated embedding $\SL_2(\Oc_K)\to \SL_2(\Rb)\times \SL_2(\Rb)$ maps $\Lambda$ onto a lattice in $\SL_2(\Rb)\times \SL_2(\Rb)$; in particular $\Lambda$ has finite index in $\SL_2(\Oc_K)$.  
\end{theorem}

\begin{proof}[Proof of Theorem~\ref{theorem:arithmeticity}] It suffices to check that $\rho(\Gamma)$ satisfies the 
criteria of Theorem~\ref{theorem:BenoistOh}. Of course it suffices to do this after a single conjugation by an element of $\SL_2(\Rb)\times\SL_2(\Rb)$; that is, after a single change of basis.

For parabolic property, we note that the  elements $\rho(\gamma_{\beta})$ and $\rho(\gamma_{\beta'})$ both stabilize $v+v'$. The variation construction shows that under the above homomorphism,  
$\rho(\gamma_{\beta})$ resp.\  $\rho(\gamma_{\beta'})$ is the image of $X^3=2X+1$ resp.\ $X^{-3}=3-2X$. The additive span $\Omega$ of these elements is of finite index in $\Oc_o$ and so $\Omega$ is a lattice in $K$. We conclude that the parabolic condition of Theorem~\ref{theorem:BenoistOh} is satisfied.  

We now claim that, after conjugating $\rho(\Gamma)$ so that the parabolic subgroup 
$P:=\langle \rho(\gamma_{\beta}), \rho(\gamma_{\beta'}\rangle$ has upper triangular form with 1 on the diagonal, $\rho(\Gamma)$ contains a matrix of the form 
$
\left(
\begin{array}{cc}
a&b\\
c&d
\end{array}
\right)
$ with $c\neq0$.  If this were not the case then the Zariski closure $G$ of the image of monodromy group in 
$\SL_2(\C)\times  \SL_2(\C)$ would lie in the group of upper triangular matrices.  We claim that $G$ is in fact all 
of $\SL_2(\C)\times  \SL_2(\C)$, a contradiction, finishing the proof of the theorem. 

To prove the claim, note that the images   $\rho(\gamma_\alpha)$ and $\rho(\gamma_{\alpha'})$ in  $\SL_2(\C)\times  \SL_2(\C)$ are
\[
\big(
(\begin{smallmatrix}
1 & -\sqrt{5}\\
0 & 1\\
\end{smallmatrix}),
(\begin{smallmatrix}
1 &\sqrt{5}\\
0 & 1\\
\end{smallmatrix})\big)  \text{ resp.\ }
\big(
(\begin{smallmatrix}
1 & 0\\
\sqrt{5} & 1\\
\end{smallmatrix}),
(\begin{smallmatrix}
1 & 0\\
-\sqrt{5} & 1\\
\end{smallmatrix})\big).
\]
It is clear that $G$ contains the  subgroup generated by the one-parameter subgroups obtained by replacing $\sqrt{5}$ by a complex variable.
These two groups generate the subgroup of $\SL_2(\C)\times  \SL_2(\C)$ that is in fact the graph of an automorphism $u$ of $\SL_2(\C)$, 
namely that assigns to $g\in\SL_2(\C)$ the transpose inverse of $g$ followed by conjugation with 
$(\begin{smallmatrix}
0 & 1\\
1 & 0\\
\end{smallmatrix})$. 
Any  connected algebraic subgroup of $\SL_2(\C)\times  \SL_2(\C)$ which strictly contains this graph is equal to $\SL_2(\C)\times  \SL_2(\C)$.
Let $U\subset \SL_2(\C)$ be the image of the one parameter group
$t\in \C\mapsto (\begin{smallmatrix}
0 & t\\
1 & 0\\
\end{smallmatrix})$. The  parabolic property shows that $G$ also contains the $U\times  U$  and so it follows that $G=\SL_2(\C)\times  \SL_2(\C)$
as asserted.  
\end{proof}

\section{The period map}
\label{section:period}

In this section we use Theorem~\ref{theorem:arithmeticity} to study various period maps associated to the Wiman-Edge pencil.

\subsection{The period map  of the Wiman-Edge pencil}
\label{subsection:period1}

We shall see that the  monodromy representation $\rho: \pi_1({\Bc^\circ})\to \SL_2(\Oc_o)$ is induced by an algebraic map from $\Bc^\circ$ to a 
quotient of a period domain $\Df$ isomorphic to $\Hf^2$ by an action of $\SL_2(\Oc_o)$. This is 
the \emph{period map}, which assigns to a curve with faithful $\Af_5$-action its Jacobian with the induced $\Af_5$-action. (Beware however, that there is {\it a priori}  no obvious relation between $\Gamma$ as a subgroup in $\PSL_2(\Z)$ and its image in $\SL_2(\Oc_o)$.)  
The quotient $Y^\circ:=\SL_2(\Oc_o)\bs \Df$ is a quasi-projective, complex-algebraic surface, called a 
\emph{Hilbert modular surface}. In this subsection we prove some properties of the period mapping; in particular we prove that it comes equipped with some extra structure. 

The two ring embeddings $\sigma, \sigma': \Oc_o\hookrightarrow\Rb$
define an algebra-isomorphism   $(\sigma, \sigma'): \Rb\otimes_\Zb \Oc_o=\Rb\otimes_\Q K\cong \Rb\oplus \Rb$.
So for a member $C$ of the Wiman-Edge pencil we have  the decomposition
\[
 H_1(C; \Rb)=\Rb \otimes_\Zb  H_1(C)\cong
\Rb\otimes_\Zb V_o(C)\otimes_{\Oc_o} E_o 
\]
Note that for  $C=C_o$, the  basis $(v,v')$ of $V_o(C_o)$ introduced in Proposition  \ref{prop:symplecticbasis} yields the  $\Rb$-basis $\{ \sigma v, \sigma v',\sigma' v, \sigma' v'\}$ for $\Rb\otimes_\Zb V(C_o)$. The anti-involution $\iota$ exchanges $v$ and $v'$ up to a common sign, but also exchanges the 
two real embeddings of $K$. In other words, it exchanges  the basis elements $\sigma v, \sigma' v'$ resp.\  $\sigma' v, \sigma v'$ up to a common sign.
The $\Oc_o$-module $V_o(C_o)$ does not have a $\Zb$-basis consisting of elements invariant under $\iota$, but the $K$-vector space 
$V_o(C_o)_\Qb$ does, namely $(v+v', X^3(v-v'))$ or $(v-v', X^3(v+v'))$, depending on whether $\iota$ exchanges $v$ and $v'$ of $v$ and $-v'$.

Let $(w,w')$ be such a basis. Assuming that $C$ is smooth, then $H_1(C)$ acquires a Hodge structure of weight $-1$ polarized by the intersection pairing. It is completely given by 
the complex subspace $F^0H_1(C)\subset H_1(C; \Cb)$. This 
is an $\Af_5$-invariant subspace that can be written as the graph of a $\Af_5$-equivariant map from the image of $w'_\Cb$ to the image of $w_\Cb$. The positivity property of the associated Hermitian form implies that there exist $\tau, \tau'\in\Hf$ such that 
$F^0H_1(C)$ is spanned by images of $\tau \sigma w+ \sigma w'$ and $\tau' \sigma' w+ \sigma' w'$. The action of $\SL_2(K)$ on $\Hf^2$  is then the standard one:
\[
\begin{pmatrix}
a &b\\
c & d\\
\end{pmatrix} (\tau, \tau')=\Big(\frac{\sigma (a)\tau+\sigma (b)}{\sigma(c)\tau +\sigma (d)},
\frac{\sigma' (a)\tau'+\sigma' (b)}{\sigma'(c)\tau'+\sigma' (d)}\Big)
\]

We prefer however to work with $(v, v')$. Let us simply write $\Df$ for  the space of  Hodge structures of weight $-1$ on $V_o(C_o)$ of the above type so that (by the above discussion) $\Df$ is  a domain isomorphic to $\Hf^2$. Then  $Y^\circ:= \SL2(\Oc_o)\backslash \Df$ is  
an algebraic surface equipped with an involution $\iota$. We note that the fixed-point set $(Y^\circ)^\iota$ of $\iota$ in $Y^\circ$ contains the image $D^\circ\subset Y^\circ$ of what corresponds to the diagonal of $\Hf^2$. Since the stabilizer of the diagonal in $\SL_2(\Oc_o)$ is $\SL_2(\Z)$, the latter  is just a copy of the $j$-line $\SL_2(\Z)\bs \Hf$. The closure $D$ of $D^\circ$ in $Y$ adds the cusp $\infty_0$, and is of course contained in $Y^\iota$. But $Y^\iota$ has other curves as irreducible components. It also contains the other cusp.

\begin{proposition}\label{prop:periodmap}
The period map  
\[\Pi^\circ: {\Bc^\circ}=\Gamma\bs\Hf\to \SL_2(\Oc_o)\bs \Df=Y^\circ\] 
is  an $\iota$-equivariant closed embedding. It extends to an  
$\iota$-equivariant morphism $\Pi: \Bc\to Y$  such that $\Pi^{-1}Y^\iota$ is the union of the $5$ points of  $\Bc\ssm \Bc^\circ$ and the point associated to the Wiman curve. The preimage of $\infty_X$  is the $C_\infty$-point of $\Bc\ssm \Bc^\circ$  and the preimage of $\infty_0$ consists of the other four points over which there 
is a singular fiber. 
\end{proposition}
\begin{proof}
The $\iota$-equivariance has already been established.
The rest of the proposition follows essentially from the  Torelli theorem, which asserts that  a smooth projective  curve can be reconstructed from  its Jacobian as a principally polarized abelian variety.  An automorphism of the latter  is up to the composition with the involution $-1$ in  $J(C)$ induced an automorphism of the curve. This automorphism is then unique unless the curve is  hyperelliptic. A curve $C$ in the Wiman-Edge  pencil is nonhyperelliptic and hence $\Aut (C)$ is identified with the group of automorphisms of $J(C)$ \emph{as a   polarized abelian variety} modulo $\pm 1$. Thus $\Pi$ can be understood as the lift  of the usual period map which also takes into account the identification of a group of automorphisms of the curve with $\Af_5$ up to inner automorphism. 
It is therefore an embedding. The cusps in $\Bc$ define stable degenerations.  This implies that $\Pi^\circ$  is proper  and extends to a morphism from $\Bc$ (which adds five points) to the Baily-Borel compactification $Y$ of $\SL_2(\Oc_o)\bs \Df$ (which adds two cusps).
\end{proof}

\subsection{Relation with the Clebsch-Hirzebruch model}
We briefly explain the relation between our Hilbert  modular surface $Y^\circ$ and another one that was investigated by Hirzebruch \cite{hirz1976} and described by him in terms of the Clebsch surface. 

Recall that the  natural map $\SL_2(\Oc)\to \SL_2(\Oc/2\Oc)\cong \SL_2(\Fb_4)$ is onto  with kernel the principal level 2 congruence subgroup $\SL_2(\Oc)[2]$ and that $\SL_2(\Oc_o)$ is the preimage of $\SL_2(\Fb_2)\subset \SL_2(\Fb_4)$. (Note that $\SL_2(\Fb_2)$  can be identified with the full permutation  group 
of the three   elements of $\Pb^1(\Fb_2)$.)
This is replicated by applying the functor $V_o$ to the chain $E\subset E_o\subset E^\vee$, for as we  observed earlier, we then
get $(2\Oc)^2\subset \Oc_o^2\subset \Oc^2$. The group $\SL_2(\Oc)[2]$ contains $-1$, hence acts on $\Df$ through $\PSL_2(\Oc)[2]:=\SL_2(\Oc)[2]/\{\pm1\}$. This action is faithful and even free
so that $Y^\circ[2]:=\SL_2(\Oc)[2]\bs \Df$ is a smooth surface. Its Baily-Borel compactification $Y^\circ[2]\subset Y[2]$ is a normal projective surface obtained by adding the five  points of $\Pb^1(\Fb_4)$.  All five points are cusp surface singularities of the same type; following  Hirzebruch they are resolved by a toroidal resolution $\hat Y[2]\to Y[2]$ for which the  preimage of each cusp is a triangle of rational curves of self-intersection $-3$.  We thus end up with
a smooth surface $\hat Y[2]$ endowed with an action of $\SL_2(\Fb_4)$, or rather, of $\PSL_2(\Fb_4)$. As we mentioned earlier, $\PSL_2(\Fb_4)\cong\Af_5$, but since we do not know whether that is a curious coincidence or that  $\PSL_2(\Fb_4)$ is naturally identified with the automorphism group of a general member of the Wiman-Edge pencil,  we prefer to make the notational distinction.

Recall that what we called in \cite{DFL} the \emph{Klein plane} and denoted by $P$, a projective plane with faithful $\SL_2(\Fb_4)$-action. It is obtained from complexification followed by projectivization of a real  irreducible representation of degree $3$  in which    $\PSL_2(\Fb_4)$ is identified  with the group of motions of a regular icosahedron. 
The 12 vertices of the icosahedron determine an $\PSL_2(\Fb_4)$-orbit in $P$ of size $6$. The blowup $\tilde P\to P$ of  this orbit is then a cubic surface with $\PSL_2(\Fb_4)$-action. It is the classical \emph{Clebsch surface}:  it  is isomorphic to the cubic surface   in the diagonal hyperplane $\sum_i z_i=0$ in $\Pb^4$ (a copy of $\Pb^3$) defined by $\sum_i z_i^3=0$, where $\Af_5$ of course acts by permuting coordinates. Since the Clebsch surface  actually comes  with an $\Sf_5$-action, so must $\tilde P$. 
The barycenters of the 20 faces of the icosahedron determine $\SL_2(\Fb_4)$-orbit in $P$ of size $10$ and appear on $\tilde P$ as its set of 
\emph{Eckardt points}, that is,  the set of points of $\tilde P$ through which pass three distinct lines on $\tilde P$. Hirzebruch proves in \cite{hirz1976} that
$\hat Y[2]$ is equivariantly isomorphic to the blowup   $\hat P\to \tilde P$  at this size $10$ orbit. It is in particular a rational surface.
It follows that our period map defines a morphism  from the base of the Wiman-Edge pencil $\Bc$ to the $\SL_2(\Fb_2)$-orbit space of $\hat P$.  
It would be worthwhile to determine its image in terms of the above construction.

\subsection{The associated family of K3 surfaces}
\label{subsection:K3}

There is another period map for the Wiman-Edge pencil, which in the terminology of Kudla-Rapoport, is of occult type. 
Recall that the Wiman-Edge pencil is realized on a quintic del Pezzo surface $S$ whose automorphism group (a copy of $\Sf_5$) preserves the pencil and induces in each member $C_t\subset S$ the $\Af_5$-action. There exists a section  $\alpha_t$ of $\omega_S^{-2}$ with divisor $C_t$. Then $\sqrt{\alpha_t}$ defines a surface
$\hat S_t$ in the total space of $\omega_S^{-1}$ (the determinant  bundle of the tangent bundle) such that the projection $\hat S_t\to S_t$ is a double cover ramified along $C_t$. Then $\hat\alpha_t:=\sqrt{\alpha_t}$ is unambiguously defined on  $\hat S_t$ as a $2$-vector and is there nowhere vanishing. 

\begin{proposition}\label{prop:K3}
The surface  $\hat S_t$ is a $K3$-surface (with an ordinary double point over every node of $C_t$). The  $\Af_5$-action on $S$  lifts uniquely to one on  $\hat S_t$ (and hence commutes with the involution) and the orthogonal complement of the $\Qb\Af_5$-embedding $H^2(S; \Qb)\hookrightarrow H^2(\hat S_t; \Qb)$ is as a $\Q\Af_5$-module isomorphic  to $3.\mathbf{1}\oplus V\oplus 2W$.  The  $3$-dimensional summand on which $\Af_5$ acts trivially has  signature $(2,1)$, and its complexification contains $H^{2,0}(\hat S_t)$.
\end{proposition}

\begin{proof}
The inverse of $\hat\alpha_t$ is a nowhere zero $2$-form. In order to conclude that $\hat S_t$ is a $K3$ surface, it suffices to shows that $H^1(\hat S_t)=0$.
If $\pi: \hat S_t\to S$ is the projection, then we have $H^1(\hat S_t)=H^1(S; \pi_*\Zb)$. The cokernel of $\Zb_S\to \pi_*\Zb$ is a rank one local system $\Lc$ on $ S\ssm C_t$ and since $H^1( S)=0$, it follows that $H^1(\hat S_t)$ embeds in $H^1(S; \pi_*\Zb/\Zb_S)=H^1_c(S\ssm C_t; \Lc)$. But  $S\ssm C_t$ is affine, and hence  $H^1_c(S\ssm C_t; \Lc)=0$.

A priori, there is a central extension of order $2$ of $\Af_5$ that lifts the $\Af_5$-action on $S$, with the nontrivial center acting as involution.
The complex line $H^{2,0}(\hat S_t)$ in $H^2(\hat S_t; \Cb)$ is preserved by this central extension with the center acting nontrivially. This implies that the central extension must be split.  In particular, the $\Af_5$-action on $S$ lifts to $\hat S_t$. It is unique, since any homomorphism  from $\Af_5$ to a cyclic group is trivial.

The $\Af_5$-representation $H^2(\hat S_t; \Qb)$ contains $H^2(S_t; \Qb)$ as a direct summand that is nondegenerate for the intersection pairing and  so the orthogonal complement, denoted   $H^2(\hat S_t; \Qb)^-$,  is a $\Q\Af_5$-module of dimension 
$22-5=17$.  We determine  its  character by computing some Lefschetz numbers. We assume here that $C_t$ is smooth, so that 
the $\Af_5$-character of $H^1(C_t; \Qb)$ is $2E_\Qb$. 

The element $(01234)$ has $2$ fixed points in $S$; this is best seen using the modular interpretation 
$(S, S\ssm C_\infty)=(\Mcbar_{0,5}, \Mc_{0,5})$. 
The fixed  points are then represented by the stable $5$-tuples on $\Pb^1$ given by $(1, \zeta_5, \zeta_5^2, \zeta_5^3, \zeta_5^4)$ and $(1, \zeta_5^2, \zeta_5^4, \zeta_5, \zeta_5^3)$, 
where $\zeta_5$ is a 5th root of unity $\not=1$.
These do not  lie on $C_\infty$ and hence not on $C_t$ for general $t$ (in fact, each of the singular dodecahedral curves contains one of them).  It follows that $(01234)$ has $4$ fixed points in $\hat S_t$.
So the trace  of $(01234)$ acting on $H^2(\hat S_t; \Qb)^-$ is $4-2=2$.
On the other hand, $(012)$ has  4  fixed points and they are represented by taking as the first three points $(1, \zeta_3, \zeta_3^2)$ and letting the last two  be arbitrary chosen in $\{ 0, \infty\}$.  So exactly two lie outside  $C_\infty$ so that $(012)$ has $2.2+2=6$  fixed points in $\hat S_t$. It follows that its trace on $H^2(\hat S_t; \Qb)^-$ is 
 $6-4=2$. 

Write 
\[
H^2(\hat S_t; \Qb)^-=a\mathbf{1} \oplus bV \oplus cW \oplus dE
\]
as $\Qb \Af_5$-modules.  The character table of $\Af_5$ shows that we have  $a + 4b+5c +6d=17$, $a-b+d=2$, $a+b-c=2$. We noted that $H^{2,0}(\hat S_t)\oplus H^{0,2}(\hat S_t)$  is a subspace of  $H^2(\hat S_t; \Cb)^-$ on which $\Af_5$ acts trivially. This subspace cannot be constant in $t$, and so we must have $a\ge 3$. We then find that the only solution is $(a,b,c,d)=(3,1,2,0)$.

We observed that the complexification  $3$-dimensional subspace of $H^2(\hat S_t; \Qb)^-$ defined by the trivial character contains 
$H^{2,0}(\hat S_t)$. This implies that its signature is $(2,1)$.
\end{proof}

With the above in hand, we are now ready to prove Theorem~\ref{theorem:uniformization}.

\begin{proof}[Proof of Theorem~\ref{theorem:uniformization}]
Since  $(H^2(\hat S_t; \Qb)^-)^{\Af_5}$ has signature $(2,1)$,  a connected component of its associated symmetric domain, which we shall denote by
$\Cf$, is of dimension one: it is copy of $\Hf$. This domain parametrizes the Hodge structures of the $K3$-surfaces with a faithful action of $\mu_2\times\Af_5$
of the type above. If $M$ is the subgroup of the orthogonal  transformations of $H^2(\hat S_t)$ of spinor norm one and acting trivially on the vectors perpendicular to $(H^2(\hat S_t; \Qb)^-)^{\Af_5}$, then our period map is defined on all of $\Bc$ and lands in the Shimura curve $M\backslash \Cf$. 
The Torelli theorem for $K3$-surfaces implies that this morphism is injective. So it must be an isomorphism. In particular,   $M\backslash \Cf$ is compact.
This means that the intersection form on  $(H^2(\hat S_t; \Qb)^-)^{\Af_5}$ does not represent zero, and that $M\backslash \Cf$ is of quaternionic type.
\end{proof}

We remark that the structure of $\Bc$ as a Shimura curve of quaternionic type cannot be induced from the period map defined by the Hodge structure $H^1(C_t)$, for the latter has cusps (and goes to a Hilbert modular surface). Simply put, the monodromy along a simple loop around a puncture is of finite order for the former and of infinite order for the latter, so the monodromy  representation $\G\to M$ is not injective.


\bigskip{\noindent
\small
Dept. of Mathematics, University of Chicago\\
E-mail: farb@math.uchicago.edu\\
\\
Yau Mathematical Sciences Center, Tsinghua University, Beijing (China), and Utrecht University\\
E-mail: e.j.n.looijenga@uu.nl

\end{document}